\documentclass[11pt,reqno,a4paper]{amsart}
\usepackage{a4wide}
\usepackage[utf8x]{inputenc}
\usepackage[T1]{fontenc}
\usepackage[english]{babel}
\usepackage{graphicx}

\usepackage{amsfonts,amssymb,amsbsy,latexsym,xspace}
\usepackage{amsmath,delarray,enumerate,calc,amsthm}
\usepackage{color}

\usepackage{ucs}
\usepackage{amsthm}
\usepackage{paralist}
\usepackage{xfrac}
\usepackage{subcaption}

\newtheorem{theorem}{Theorem} % 1st argument is your name for it
\newtheorem{lemma}{Lemma}[section]     % 2nd argument is what is printed
\newtheorem{corollary}[lemma]{Corollary}
\newtheorem*{corollary*}{Corollary}
\newtheorem{proposition}[lemma]{Proposition}
\newtheorem{definition}[lemma]{Definition}
\theoremstyle{definition}
\newtheorem{example}{Example}

\newtheorem{remark}[lemma]{Remark}
\newtheorem{hypothesis}{Hypothesis}

\renewcommand{\mod}{\text{ mod }}

\newcommand{\R}{\mathbb{R}}
\newcommand{\N}{\mathbb{N}}
\newcommand{\Z}{\mathbb{Z}}

\newcommand{\T}{{\mathbb{T}}}
\newcommand{\I}{{\mathbb{I}}}

\newcommand{\cD}{\mathcal{D}}
\newcommand{\cE}{\mathcal{E}}
\newcommand{\cG}{\mathcal{G}}

\newcommand{\cP}{\mathcal{P}}
\newcommand{\cS}{\mathcal{S}}

\newcommand{\cM}{\mathcal{M}}

\newcommand{\GA}{{\rm GA}}
\newcommand{\ellss}{{\ell^{\rm ss}}}
\newcommand{\id}{{\rm id}}
\newcommand{\supp}{\operatorname{supp}}

\newcommand{\dimH}{\operatorname{dim}_H}

\newcommand{\dom}{\mathsf{D}}

\newcommand{\CLip}{C^{1+{\rm Lip}}}

\numberwithin{equation}{section}

\title{Chaotically driven sigmoidal maps}
\author{Gerhard Keller and Atsuya Otani}
\address{Department Mathematik, Universit\"at Erlangen-N\"urnberg, 91058 Erlangen, Germany}
\email{keller@math.fau.de, otani@math.fau.de}

\thanks{These notes profited enormously from several inspiring workshops in the framework of the DFG Scientific Network ``Skew Product Dynamics and Multifractal Analysis'' organised by Tobias Oertel-J\"ager. The research was funded by the German RAesearch Foundation (DFG Ke 514/8-1).}
\date{\today}

\begin{document}

%\tableofcontents
\begin{abstract}
We consider skew product dynamical systems 
$f:\Theta\times\R\to\Theta\times\R, f(\theta,y)=(T\theta,f_\theta(y))$
with a (generalized) baker transformation $T$ at the base and
uniformly bounded increasing $C^3$ fibre maps $f_\theta$ with negative Schwarzian derivative.
Under a partial hyperbolicity assumption that ensures the existence of strong stable fibres for $f$ we prove that the presence of these fibres restricts considerably the possible structures of invariant measures - both topologically and measure theoretically, and that this finally allows to provide a ``thermodynamic formula'' for the Hausdorff dimension of set of those base points over which 
the dynamics are synchronized, i.e. over which  the global attractor consists of just one point.
\end{abstract}

\maketitle

\section{Introduction}
We consider skew product dynamical systems with monotone one-dimensional fibre maps:
\begin{equation*}
f:\Theta\times\R\to\Theta\times\R,\quad f(\theta,y)=(T\theta,f_\theta(y))\ ,
\end{equation*}
where $T:\Theta\to\Theta$ is a suitable invertible map (often a homeomorphism), and where the fibre maps $f_\theta:\R\to\R$ are uniformly bounded increasing $C^3$-maps with negative Schwarzian derivative
\begin{equation*}
\cS f_\theta:=\frac{f_\theta'''}{f_\theta'}-
\frac{3}{2}\left(\frac{f_\theta''}{f_\theta'}\right)^2<0
\end{equation*}

When the base transformation $T:\Theta\to\Theta$ is just a bi-measurable bijection of a measurable space, the following measure theoretic classification is available. The global attractor $\GA:=\bigcap_{n=0}^\infty f^n(\Theta\times\R)$
is the region between a lower invariant graph $\varphi^-$ and
an upper one $\varphi^+$. %, which are both constructed as pull-back limits.
There is a third relevant graph $\varphi^*$ that satisfies $\varphi^-\leqslant\varphi^*\leqslant\varphi^+$.
More precisely we have the following: Denote by $\cM_T(\Theta)$ the set of all $T$-invariant probability measures on $\Theta$ and by $\cE_T(\Theta)$ the ergodic ones among them. For $\nu\in\cM_T(\Theta)$ and $i\in\{-,\ast,+\}$ let
\begin{equation*}
\lambda^i(\nu):=\int_\Theta\log f_\theta'(\varphi^i(\theta))\,d\nu(\theta)\ ,
\end{equation*}
and for any measurable map $\varphi:\Theta\to\R$ denote
$\nu_\varphi:=\nu\circ(\id,\varphi)^{-1}$.

\begin{proposition}\cite[Proposition 2.21, based on \cite{Jager2003}]{Otani2015} \label{prop:classification-general}
$\varphi^-$ and $\varphi^+$ are measurable $T$-invariant graphs defined everywhere as pullback-limits and $\GA=\{(\theta,y)\in\Theta\times\R: \varphi^-(\theta)\leqslant y\leqslant\varphi^+(\theta)\}$.
 Furthermore, there is a  measurable graph $\varphi^\ast$ defined everywhere\footnote{The function $\varphi^*$ can be chosen e.g. as one of the functions $\varphi^{*-}$ or $\varphi^{*+}$ from \cite[Definition 2.16]{Otani2015}. We note that there might be points where $\varphi^{*-}>\varphi^{*+}$, but these points form a null set for each invariant measure. Definition 2.19 of a separating Graph in \cite{Otani2015} must be changed slightly in order to deal with these exceptional points: at such points a separating graph need not lie above $\varphi^{*-}$ and below $\varphi^{*+}$; e.g. one can set it to be equal to $\varphi^-$ on this set.}, satisfying $\varphi^-\leqslant\varphi^\ast\leqslant\varphi^+$ and such that for each
$\nu\in\cE_T(\Theta)$ exactly one of the following three cases occurs:
\begin{compactenum}[(i)]
\item $\varphi^-=\varphi^\ast=\varphi^+$ $\nu$-a.s. and $\lambda^\pm(\nu)\leqslant0$.
\item $\varphi^-=\varphi^\ast<\varphi^+$ $\nu$-a.s. and $\lambda^+(\nu)<0=\lambda^-(\nu)$, or $\varphi^-<\varphi^\ast=\varphi^+$ $\nu$-a.s. and
$\lambda^-(\nu)<0=\lambda^+(\nu)$.
\item $\varphi^-<\varphi^\ast<\varphi^+$ $\nu$-a.s., $\lambda^-(\nu)<0$,$\lambda^+(\nu)<0$, $\lambda^\ast(\nu)>0$, and $\varphi^\ast$ is $\nu$.a.s. $T$-invariant.
\end{compactenum}
In all three cases holds:
\begin{compactenum}[-]
\item Each graph which is $\nu$-a.s. $T$-invariant equals $\varphi^-$, $\varphi^\ast$ or $\varphi^+$ $\nu$-a.s.
In particular, $\nu_{\varphi^-},\nu_{\varphi^*}$ and $\nu_{\varphi^+}$ are the only $f$-invariant ergodic probability measures that project to $\nu$.
\item For $\nu$-a.e. $\theta\in\Theta$ and every $y \in [ \varphi ^- (\theta ) , \varphi ^+ ( \theta) ] \setminus \{ \varphi^\ast(\theta) \}$
\begin{equation}
\begin{split}\label{eq:separation}
\lim_{n\to\infty}|f_\theta^n(y)-\varphi^-(T^n\theta)|=0
&\text{ if }y<\varphi^\ast(\theta)\ ,\\
\lim_{n\to\infty}|f_\theta^n(y)-\varphi^+(T^n\theta)|=0
&\text{ if }y>\varphi^\ast(\theta)\ .
\end{split}
\end{equation}
\end{compactenum}
\end{proposition}

Forward Lyapunov exponents play an important role in this note. They are defined by
\begin{equation}\label{eq:lyap-def}
\lambda_{(\theta,y)}
:=
\limsup_{n\to\infty}\frac{1}{n}\log(f_\theta^n)'(y)
\end{equation}

\begin{corollary}\label{coro:exponent}
For $\nu$-a.e. $\theta\in\Theta$ and every $y\in\R$, 
\begin{equation*}
\lambda_{(\theta,y)}
=
\lim_{n\to\infty}\frac{1}{n}\log(f_\theta^n)'(y)
=
\begin{cases}
\lambda^-(\nu)&\text{ if }y<\varphi^\ast(\theta)\\
\lambda^\ast(\nu)&\text{ if }y=\varphi^\ast(\theta)\\
\lambda^+(\nu)&\text{ if }y>\varphi^\ast(\theta)\ .
\end{cases}
\end{equation*}
\end{corollary}
\noindent The \emph{proof} is deferred to the appendix.

If $T:\Theta\to\Theta$ is a homeomorphism of a compact metric space, then $\varphi^-$ is upper semi-continuous, $\varphi^+$ is lower semi-continuous by construction, and the global attractor $\GA$ is compact.
The set $P\subseteq\Theta$ of points, at which $\varphi^-$ and $\varphi^+$ coincide, is a $G_\delta$-set. It is always dense, if it is non-empty and if the base system is minimal, e.g. if it is an irrational rotation (see Subsection~\ref{subsec:qpf} for more on this).

\subsection{The contribution of this paper}
Here we consider the case of a hyperbolic base transformation $T$, where the presence of strong stable fibres for the map $f$ serves to restrict the possibilities how the set $P$ and the graph of $\varphi^*$ can look like in the three cases of Proposition~\ref{prop:classification-general}. 
In order to concentrate on the interesting phenomena and keep the technicalities at a minimum, we will only consider (generalized) baker transformations at the base\footnote{Of course they are not continuous in a strict sense, but their discontinuities are mild enough that $\varphi^-$ and $\varphi^+$ are still semi-continuous under our assumptions.}. 

A consequence of the main structure results for the set $P$ (Theorems~\ref{theo:top-class} - \ref{theo:B-class}) is the applicability of thermodynamic formalism to derive a formula for the Hausdorff dimension of $P$ in Theorem~\ref{theo:dimension}, using results from \cite{Otani2015}.

\subsection{Relations to weak generalized synchronization}
The set $P$ consists of those base points, over which the system synchronizes trajectories with different initial values on the same fibre. This phenomenon can be understood as a variant of \emph{weak generalized synchronization} 
\cite{Pyragas1996,Hunt98,Singh2008,KJR2012}. In the terminology of those references, 
$\phi^+|_P=\phi^-|_P$ is the (partially defined) response function coupling the dynamics on the fibres to the chaotic base.
The results of this paper elucidate the highly complex fine structure of this response mechanism.

\subsection{Relations to Weierstrass graphs} If $T:[0,1)^2\to[0,1)^2$ is the baker map and if instead of branches with negative Schwarzian derivative one studies affine branches $f_\theta(y)=\lambda y+\cos(2\pi x)$ with $\lambda\in(1/2,1)$ (where $\theta=(\xi,x)$),
it turns out that $\varphi^+=\varphi^-$ everywhere and that this is the graph of a classical Weierstrass function - a Hölder-continuous but nowhere differentiable function. These functions and their generalizations have received much attention during the last decades, culminating in the recent result that the graph of this function has Hausdorff dimension $2+\frac{\log\lambda}{\log 2}$ for all $\lambda\in(1/2,1)$, i.e. in the range of parameters for which the skew product system $f$ is partially hyperbolic \cite{BBR2013,Shen2015}.

Related, but less complete results were also obtained for for generalisations where the cosine function is replaced by other suitable functions, where the baker map is replaced by some of its nonlinear variants, and/or where the slope $\lambda$ of the contracting branches is also allowed to depend on $x$, see e.g. \cite{Urbanski89,Ledrappier92,Otani2015}. In all these cases, however, the fibre maps remain affine. Other results deal with the local Hölder exponents of these graphs \cite{Bedford89, Bedford89h}.

In this note we replace the affine branches from the Weierstrass case by negative Schwarzian branches in a way that allows simultaneously contraction and expansion in the fibres. Our primary goal is to understand the interplay between topological and measure theoretic properties of the graphs of $\varphi^-,\varphi^*$ and $\varphi^+$. In the end it turns out that this is the key to a ``thermodynamic formula'' for the Hausdorff dimension  of set of those base points over which the global attractor is pinched, i.e. consists of just one point.

\subsection{Relations to quasiperiodically forced systems}
\label{subsec:qpf}
Let $T:\T\to\T$ be an irrational circle rotation. Such quasiperiodically forced systems were studied by Jäger in \cite{Jager2003,Jager2009}. He gives the following result \cite[Theorem 4.2 and Corollary 4.3]{Jager2003} under the aditional assumption that $(\theta,y)\mapsto f_\theta'(y)$ is continuous:
There are three possible cases:
\begin{compactenum}[(i)]
\item There exists only one invariant graph $\varphi$, and $\lambda(\varphi)\leqslant0$.
If $\lambda(\varphi)<0$, then $\varphi$ is continuous. Otherwise its Lebesgue equivalence class contains at least an upper and a lower semi-continuous representative.
\item There exist two invariant graphs $\varphi$ and ${\varphi^\ast}$ with $\lambda(\varphi)<0$ and $\lambda({\varphi^\ast})=0$.
The upper invariant graph is upper semi-continuous, the lower invariant graph lower semi-continuous. If $\varphi$ is not continuous and ${\varphi^\ast}$ (as an equivalence class) is only semi-continuous in one direction, then 
$\supp(m_\varphi)=\supp(m_{\varphi^\ast})$, i.e. the Lebesgue-essential closures of the two graphs coincide.
\item There exist three invariant graphs $\varphi^-\leqslant{\varphi^\ast}\leqslant\varphi^+$ with $\lambda(\varphi^\pm)<0$ and
$\lambda({\varphi^\ast})>0$.
${\varphi^\ast}$ is continuous if and only if $\varphi^+$ and $\varphi^-$ are continuous. Otherwise $\varphi^-$ is at least lower semi-continuous and $\varphi^+$ is at least upper semi-continuous. If ${\varphi^\ast}$ is neither upper nor lower semi-continuous, then 
$\supp(m_{\varphi^-})=\supp(m_{\varphi^\ast})=\supp(m_{\varphi^+})$.
\end{compactenum}
All three cases occur. For case~(iii) this was proved in \cite{Jager2009}.

In this note we replace the quasiperiodic base by a chaotic one, namely by a baker transformation, and we prove an analoguous but rather different classification of the types of invariant graphs.

\section{Negative Schwarzian branches driven by a baker map: Model and results}
\label{sec:results}

\subsection{Preliminaries}

In this note we concentrate on the simplest hyperbolic case where $T:\T^2\to\T^2$ is a  (generalized) baker map with discontinuity at $a\in(0,1)$, i.e.
\begin{equation*}
T(\xi,x)
=
\begin{cases}
(\tau(\xi),ax)&\text{ if }\xi\in[0,a)\\
(\tau(\xi),a+(1-a)x)&\text{ if }\xi\in[a,1)
\end{cases}
\quad\text{with}\quad
\tau(\xi)
=
\begin{cases}
a^{-1}\xi&\text{ if }\xi\in[0,a)\\
(1-a)^{-1}(\xi-a)&\text{ if }\xi\in[a,1)
\end{cases}.
\end{equation*}
This map is bijective but not continuous. Nevertheless the pull-back construction mentioned in Proposition~\ref{prop:classification-general} yields semi-continuous $\varphi^-$ and $\varphi^+$ in some cases of interest. We make the following general assumption:

\begin{hypothesis}\label{hyp:special-ass}
$\Theta=\T^2$, $T:\T^2\to\T^2$ is a generalized baker map,
$f_\theta$ depends on $\theta=(\xi,x)$ only through $x$, so that we can write $f_\theta=f_x$,
and $x\mapsto f_x$ is continuous on $\T $, where $\T$ is the one-dimensional torus.
Recall also that $f_x'>0$ and $\cS f_x<0$.
\end{hypothesis}

\begin{remark}
The assumptions that the fibre maps depend on $\theta$ only through $x$ is certainly a severe restriction, because it prevents a direct extension of the present results to models where $T$ is a hyperbolic diffeomorphism of $\T^2$. In the much simpler case of multiplicatively forced concave fibre maps with a common fixed point, this problem could be overcome in the following way \cite[Remark 4 and Proposition 1.1]{Keller2014}: For baker maps at the base, a multiplicative driving function which depends Hölder-continuously on $\theta$ is cohomologous to one that depends only on $x$. Although the invariant graphs of the systems defined by these two driving functions are not identical, they can be sufficiently well compared to transfer knowledge about one of them to the other one. 
In \cite[Example 2 and Theorem 2.6]{Keller2014} it is discussed how this idea can  be applied when the base map is an Anosov surface diffeomorphism.

For the more general fibre maps studied here one might try to compare a system with fibre maps depending on $\theta$ to one with fibre maps depending only on $x$, and to transfer knowledge
(in the sense of Theorems~\ref{theo:top-class} -~\ref{theo:dimension} below) about the latter system to the original one.
\end{remark}

\begin{lemma}	\label{lem:phi_depend_x}
If Hypothesis~\ref{hyp:special-ass} is satisfied,
then $\varphi^-$ and $\varphi^+$ are semi-continuous functions from $\T^2$ to $\R$ that depend on $\theta=(\xi,x)\in\T^2$ only through $x$, i.e. $\varphi^\pm(\xi,x)=\phi^\pm(x)$.\end{lemma}

\begin{proof}
Here is the proof for $\varphi^+$: Fix $M>0$ such that $f_\theta(y)\leqslant M$ for all $(\theta,y)\in\T^2\times\R$, and define
$\psi_k^+(\theta)=f_{T^{-k}\theta}^k(M)$. Then
$M\geqslant\psi_1^+\geqslant\psi_2^+\searrow\varphi^+$ pointwise, so it suffices to prove that all $\psi_k^+:\T^2\to\R$ are continuous.
But $\psi_k^+(\xi,x)=f^k_{\tau^kx}(M)$, and $\tau^k:\T \to\T $ is a continuous map of the $1$-torus. 
\end{proof}
In addition, we need a uniform partial hyperbolicity assumption.
\begin{hypothesis}\label{hyp:weak-contraction}
There are closed intervals $ I\subseteq J \subseteq \R $ such that $ f ( \T ^2 \times J ) \subseteq \T ^2 \times I ^\circ $, $I\subseteq J^\circ$, and
\[
	\inf _{ ( x , y) \in \T   \times J } \tau ' ( x ) \cdot f _{\tau ( x )} ' ( y ) > 1 .
\]

\end{hypothesis}

\begin{example}	\label{ex:main_example}
Let $ \tau (x) = 2 x \mod 1 $, $	f _x(y) = 	\arctan ( r y ) +  \varepsilon \cos ( 2 \pi x ) $ and $ J = [ - M , M ]$, where $ r = 1.1 $ and $ \varepsilon \in [ 0 , 0.1] $.
Then we can choose $M=0.86$ and $I=[-0.858,0.858]$.

More generally, for given $\varepsilon$
denote by $U_\varepsilon$ the set of parameters $(M,r)$, for which
Hypothesis~\ref{hyp:weak-contraction} is satisfied.
Then it is not hard to show that, for $|\varepsilon|\leqslant0.21$,  the parameter $(1,1)$ belongs to the closure of $U_\varepsilon$, and
for $|\varepsilon|\leqslant0.1$, $U_\varepsilon$ contains the triangle
determined by the inequalities $1\leqslant r\leqslant M+0.24$ and $M\leqslant0.98$.
\end{example}

\begin{equation}
	D f _{(\xi, x, y)} = \left( \begin{matrix}
		\tau ' (\xi)	& 0 & 0 \\
		0 &  \sigma(\xi)  & 0 \\
		0 & \frac{\partial f_x}{\partial x}(y) &  f _x ' (y) 
	\end{matrix} \right) ,	\label{eq:DF}
\end{equation}
where 
$\sigma(\xi)=a$ if $\xi\in[0,a)$ and $\sigma(\xi)=1-a$ if $\xi\in [a,1)$.
(In fact, $\sigma(\xi)=1/\tau'(\xi)$.)       

By Hypothesis~\ref{hyp:weak-contraction}, the second diagonal entry is the eigenvalue  responsible for the strong stable directions in the skew product system.
In order to guarantee the existence of sufficiently regular strong stable fibres, we impose the following regularity conditions:

\begin{hypothesis}\label{hyp:regularity}
The map $ \T \times J \rightarrow \R $, $(x,y)\mapsto f_x(y)$, is of class $C^2$.   
%exist and are continuous:
%\begin{itemize}
% \item $ (x, y) \mapsto f _x ' (y) $,
% \item  $ ( x, y) \mapsto \frac{\partial}{\partial y}\log f_x'(y) $, 
% \item  $ (x, y) \mapsto \beta(x,y)$, where $ \beta ( x , y ) = \frac{\partial}{\partial x} f _x (y)  $, and
% \item  $ (x, y) \mapsto \frac{\partial}{\partial y}\log\beta(x,y)$
% % \item $ (x, y) \mapsto \frac{\partial}{\partial x} ( x \mapsto \log f_x' (y))$.
%\end{itemize}
\end{hypothesis}

Along classical lines one proves the following lemma (see the appendix):
\begin{lemma}\label{lemma:shortX3}
Under Hypothesises \ref{hyp:special-ass}, \ref{hyp:weak-contraction} and \ref{hyp:regularity}, there exists a bounded $Df$-equivariant strong stable subbundle $X$ of the tangent space to $\T^2\times J$. It satisfies
$Df\cdot X=\sigma\cdot(X\circ f)$, and for each $\xi\in\T$, the map $(x,y)\mapsto X(\xi,x,y)$ is continuous on $\T\times J$.
(Indeed, $X_1=0$, $X_2=1$ and all information is carried by $X_3:\T^2\times J\to\R$, which is given explicitly in (\ref{eq:X_3}).) Moreover,
\begin{equation}\label{eq:X3deriv-y-bound}
 	\sup _{(\xi, x, y) \in \T ^2 \times J } \left| \frac{\partial X _3}{\partial y} (\xi, x, y) \right| < \infty 
\end{equation}
and
\begin{equation}\label{eq:X3deriv-x-bound}
 	\sup _{(\xi, x, y) \in \T ^2 \times J } \left| \frac{\partial X _3}{\partial x} (\xi, x, y) \right| < \infty\ .
\end{equation}
\end{lemma}

We often will identify points $ x \in \T = \R / \Z $ with points $ x \in \I:=[0,1] $.
In this case we also identify $ (\xi, x) \in \T ^2 $ with $ ( \xi , x ) \in \I^2 $.
When we consider strong stable fibres, that we are going to introduce now, the space $ \T $ for $ x $ is always considered as $\I$, so that effectively we are looking at particular local strong stable fibres.

\begin{definition}	\label{def:l_ss}
Let $X_3:\T^2\times J\to\R$  be the function from Lemma \ref{lemma:shortX3}.
Given $ (\xi, x, y) \in \I^2 \times J$, we define the strong stable fibre 
through this point
as the unique maximal solution $ \ellss_{(\xi,x,y)} : \dom _{(\xi, x, y)} \rightarrow J $ of the initial value problem
\begin{equation}
\left\{
\begin{array}{rcl}
	\ellss _{(\xi, x, y)} (x) & = & y  \\
	(\ellss _{(\xi, x, y)} ) ' {(u)}	&  =&	X _3 \left( \xi, u,  \ellss_{(\xi,x, y)} (u) \right) 
\end{array} 
\right.	\label{eq:IVP}
\end{equation}
for $ u \in \dom _{(\xi, x, y)} $, where $ \dom _{(\xi, x, y)} \subseteq \I $ is the maximal interval around $ x \in \I $ for which the unique solution exists. (It will turn out that $\dom_{(\xi,x,y)}=\I$
for most ``interesting'' points $(\xi,x,y)$.)
\end{definition}
\noindent Observe that (\ref{eq:X3deriv-y-bound}) allows to apply the Picard-Lindelöf theorem to (\ref{eq:IVP}), 
and together with \eqref{eq:X3deriv-x-bound} it also implies that all strong stable fibres are of class~$\CLip$.

\subsection{The classification}

From now on, we always assume Hypothesis \ref{hyp:special-ass},  \ref{hyp:weak-contraction} and \ref{hyp:regularity}.

\begin{definition}
Denote by $\cG\varphi^\pm$ the graphs of $\varphi^+,\varphi^-:\T^2\to I$, by $\overline{\cG\varphi^\pm}$ their closures in $\T^2\times I$, and by
$\widetilde{\cG\varphi^\pm}$ their filled-in closures in $\T^2\times I$, i.e.
\[
	\widetilde{\cG\varphi^\pm} = \left\{ (\theta , y )  \in \T^2 \times I : \exists y_1, y_2 \in \overline{\cG\varphi ^\pm} \mbox{ s.t. } y_1 \leqslant y \leqslant y _2 \right\} .
\]
Let $P:=\{\theta\in\T^2: \varphi^+(\theta)=\varphi^-(\theta)\}$ 
be the set of \emph{pinched points} of the system, i.e. the set of points $\theta=(\xi,x)$ where $\varphi^-(\theta)=\varphi^+(\theta)$,   
and
denote by $C^\pm\subseteq\T^2$ the set of continuity points of $\varphi^\pm$. (As $\varphi^+$ and $\varphi^-$ are semi-continuous, they are both Baire 1 functions, and hence $C^+$, $C^-$ and also their intersection are dense $G_\delta$-sets.)
\end{definition}

\begin{theorem}[Topological classification]\label{theo:top-class}\quad\\
The sets $\overline{\cG\varphi^\pm}$ and $\widetilde{\cG\varphi^\pm}$ are forward $f$-invariant.
There are only two possibilities:
\begin{enumerate}[A)]
\item $\widetilde{\cG\varphi^+}\cap\widetilde{\cG\varphi^-}=\emptyset$, in particular $P=\emptyset$, or
\item $P=C^+\cap C^-$, and this is a dense $G_\delta$-subset of $\T^2$.
Let $\cP:=\overline{\cG\varphi^\pm|_P}$. Then $\cP\subseteq
\overline{\cG\varphi^+}\cap\overline{\cG\varphi^-}$, $f(\cP)\subseteq\cP$, and $\pi(\cP)=\T^2$,
where $\pi:\T^2\times \R\to\T^2$ is the canonical projection.
\end{enumerate}
If there is any $\nu\in\cE_T(\T^2)$ with $\lambda^*(\nu)<0$, then case~A is excluded.
\end{theorem}
For a probability measure $\nu$ on $\T^2$ denote by $\nu^{(1)}$ and $\nu^{(2)}$ the first and second marginal of $\nu$, respectively, i.e. $\nu^{(i)}=\nu\circ\pi_i^{-1}$ $(i=1,2)$. Observe that each Gibbs measure $\nu\in\cM_T(\T^2)$ has a nice product structure in the sense that
$\nu\approx\nu^{(1)}\otimes\nu^{(2)}$, and, even more, if $\nu=\int _\T\nu_\xi\,d\nu^{(1)}(\xi)$, then $\nu_\xi\approx\nu^{(2)}$ for $\nu^{(1)}$-a.a. $\xi$, and that $\supp(\nu^{(i)})=\T$ $(i=1,2)$ for such $\nu$. Only these properties of Gibbs measures are really required for the next two theorems.
Denote by $\nu_{\varphi^*}=\int_\T(\nu_{\varphi^*})_\xi\,d\nu^{(1)}(\xi)$ the decomposition of $\nu_{\varphi^*}$ into conditional measures on the hypersurfaces $\{\xi\}\times\T\times I$. 

\begin{theorem}[Fine structure in case~A]
\label{theo:A-class}
Suppose case~A of Theorem~\ref{theo:top-class}. Let $\nu\in\cE_T(\T^2)$ be a Gibbs measure. Then for $\nu^{(1)}$-a.a. $\xi$ there is a unique strong stable fibre $\gamma_\xi^*$ in the $\xi$-plane which supports the conditional measure $(\nu_{\varphi^*})_\xi$, and exactly one of the following cases occurs:
\begin{enumerate}[A1)]
\item[A$_\nu1$)] $\nu_{\varphi^+}$ and $\nu_{\varphi^-}$ are the only $f$-invariant measures on $\widetilde{\cG\varphi^+}$ and $\widetilde{\cG\varphi^-}$, respectively, which project to $\nu$.
% There is no further ergodic invariant measure that projects to $\nu$. 
For $\nu^{(1)}$-a.a. $\xi$, $\gamma^*_\xi$ is disjoint from $\widetilde{\cG\varphi^+}$ and from $\widetilde{\cG\varphi^-}$. Moreover,
$\lambda^\pm(\nu)<0$ while $\lambda^*(\nu)>0$.
\item[A$_\nu2^+$)] All $f$-invariant measures which project to $\nu$ are supported by $\widetilde{\cG\varphi^+}$ and $\widetilde{\cG\varphi^-}$.
$\nu_{\varphi^*}$ is supported  by the lower bounding graph $\varphi^{+-}:\theta\mapsto\inf\left(\widetilde{\cG\varphi^+}\right)_\theta$ of
$\widetilde{\cG\varphi^+}$.
For $\nu^{(1)}$-a.a. $\xi$, $\gamma_\xi^*=\cG\varphi^{+-}$.
Moreover, either
\begin{compactenum}[a)]
\item $\varphi^{+-}<\varphi^+$ $\nu$-almost surely and
$\lambda^\pm(\nu)<0$ while $\lambda^*(\nu)>0$, or
\item $\varphi^{+-}=\varphi^+$ $\nu$-almost surely and
$\lambda^-(\nu)<0$ while $\lambda^+(\nu)=\lambda^*(\nu)=0$.
\end{compactenum}
\item[A$_\nu2^-$)] As A$_\nu2^+$), but with roles of $\varphi^+$ and $\varphi^-$ interchanged.
\end{enumerate}
\end{theorem}

Denote by $m$ Lebesgue measure on $\T$. The following is an immediate corollary to Theorem~\ref{theo:A-class}.
\begin{corollary*}
Assume case~A of Theorem~\ref{theo:top-class}. Then exactly one of the following two possibilities occurs:
\begin{compactenum}[A1)]
\item[A1)]
For $m$-a.e.~$\xi$, the topological support of the  conditional measure $((m^2)_{\varphi^*})_\xi$  
is disjoint from $\widetilde{\cG\varphi^-}$ and $\widetilde{\cG\varphi^+}$
(it lies ``between'' these two sets).
\item[A2)] The measure $(m^2)_{\varphi^*}$ is supported by
$\cG\varphi^{+-}$ or by $\cG\varphi^{-+}$.
This graph depends only on the coordinate $x$, is of class $\CLip$, and its restriction to the $\xi$-plane coincides for $m$-a.a. $\xi$ with the strong stable fibre $\gamma_\xi^*$.
\end{compactenum}
\end{corollary*}

\begin{theorem}[Fine structure in case~B]
\label{theo:B-class}
Assume case~B of Theorem~\ref{theo:top-class}.
Then exactly one of the following two possibilities occurs:
\begin{compactenum}[B1)]
\item[B1)] There is some $ \theta \in \T ^2$ such that $\# \{ y \in I : ( \theta , y )  \in \cP \} > 1$. Moreover,
$ \nu ( P ) = 1 $ and $\lambda^*(\nu)\leqslant0$ for any Gibbs measure $ \nu $.
\item[B2)] $\cP$ is the graph of a $ f $-invariant continuous function
$\hat\varphi:\T^2\to I$, which depends only on the variable $x$, i.e. $\hat\varphi(\xi,x)=\hat\phi(x)$.

Moreover,
$\hat\varphi = \varphi ^\ast$ $\nu$-a.s.~for any Gibbs measure $\nu$, and
$\lambda^*(\nu)\leqslant0$ if $\nu(P)=1$.
Regarding the regularity of $\hat\phi$, there are two possibilities:
\begin{compactenum}[a)]
\item[B2-i)] $\hat\phi$ is not everywhere differentiable with bounded derivative, or
\item[B2-ii)] $\hat\phi$ is of class $\CLip$. In this case $\hat\phi=\ellss_{(\xi,x,\hat\phi(x))}$ for all $(\xi,x)\in\T^2$. 
\end{compactenum} 
\end{compactenum} 
If there is some Gibbs measure $\nu$ with $\nu(P)=0$, then B2-ii occurs.
\end{theorem}

\begin{corollary*}
%\label{cor:B-class}
Assume case~B of Theorem~\ref{theo:top-class}.
\begin{compactenum}[a)]
\item If $m^2 (P)=0$, then $\cP$ is the graph of a $\CLip$-function over $\T^2$, that depends only on the stable coordinate.
\item If $\varphi^+$ or $\varphi^-$ is continuous, then its graph coincides with the set $\cP$, and case~B1 is excluded. If, in this situation, $m^2(P)=0$, then
$\varphi^+$ or $\varphi^-$ is $\CLip$, respectively.
\end{compactenum}
\end{corollary*}

The following theorem is an immediate consequence of Theorems~\ref{theo:top-class} and~\ref{theo:B-class}.

\begin{theorem}
If there exists a Gibbs measure $\nu$ with $\nu (P)=0$, then $\cP=\emptyset$ or $\cP$ is the graph of a $\CLip$-function over $\T^2$, that depends only on the stable coordinate.
%If $\cP\neq\emptyset$ is not the graph of a $\CLip$-function over $\T^2$, then $\nu(P)=1$ for any Gibbs measure $\nu$.
\end{theorem}

\noindent
\textbf{Example~\ref{ex:main_example}} (continued).
We fix $(M,r)=(0.86,1.1)$ and consider various $\varepsilon\in[0,0.1]$.
\begin{compactenum}[a)]
\item If $0\leqslant\varepsilon\leqslant0.018$, then there is
$y_\varepsilon\in(0,M)$ such that
$f_x([y_\varepsilon,M])\subseteq(y_\varepsilon,M)$,
$f_x([-M,-y_\varepsilon])\subseteq(-M,-y_\varepsilon)$, and
$\sup_{0\leqslant x\leqslant 1,|y|\geqslant y_\varepsilon}(f_x\circ f_{\tau x})'(y)<0.99$  for all $x\in\T$.
Hence $\varphi^+$ and $\varphi^-$ are continuous and have disjoint graphs,
in particular this is case~A1. For $\varepsilon=0.018$ this is illustrated in 
Figure~\ref{fig:graphs}a.
\item  If $0.019\leqslant\varepsilon\leqslant0.1$, then the branch
$f_{0}$ over the fixed point
$\theta=(0,0)$ of $T$ has a unique fixed point $y_0=f_{0}(y_0)>0.6$.
%As $\cS f_{0}<0$, it follows that
Hence $y_0=\phi^-(0)=\phi^+(0)$,
%(Proposition~\ref{prop:classification-general}),
so that case~A of Theorem~\ref{theo:top-class} is excluded.

\begin{compactenum}[i)]
\item If $0.019\leqslant\varepsilon\leqslant0.037$, then $f_x\circ f_{\tau x} (0.3)>0.3$
and 
$\sup_{0\leqslant x\leqslant 1,|y|\geqslant 0.3}(f_x\circ f_{\tau x})'(y)<0.997$ 
 for all
$x\in\T$. 
Hence $\phi^+>0.3$, and $\phi^+$ is continuous. As case~A is excluded, $\cP$ coincides with the graph of $\varphi^+$, so that this is case~B2.
The branches $f_{1/3}=f_{2/3}$ over these $2$-periodic points have three fixed points, among them a unique negative one, call it $y_1$,
which
is stable. Hence $\phi^-(1/3)=y_1<0$ (the same holds for $y=2/3$), so that $\phi^-(1/3)\neq\phi^+(1/3)$.
On the other hand, $\nu\{\varphi^-=\varphi^+\}=\nu(P)=1$ and $\lambda^\ast(\nu)\leqslant 0$ for any $\nu\in\cM_T(\T^2)$ with $0\in\supp(\nu)$\,\footnote{\label{foot:proof1}	The claim is shown as follows.
Since $\phi^-(0)=y_0>0.6$, there is some $\delta>0$ such that
$\phi^-(x)>0.3$ for all $x \in [0,\delta]$. By Poincaré's recurrence
theorem, for $\nu$-a.a. $\theta=(\xi,x)$, there is some $n\in\N$ such
that $T^{-2n}\theta\in\T\times [0,\delta]$ so that
\[\varphi^-(\theta)=(f_{\tau x}\circ f_{\tau ^2
x})\circ\dots\circ(f_{\tau^{2n-1}x}\circ f_{\tau
^{2n}x})\circ\phi^-(\tau^{2n}x)>0.3.\]
As all branches $f_x\circ f_{\tau x}$ leave the strip $\{y>0.3\}$ invariant and are uniformmly contracting on this strip, it follows that $\varphi^-(\theta)=\varphi^+(\theta)$ for $\nu$-a.a. $\theta$.}, 
in particular
$\varphi^->0.3$ Lebesgue-a.e.
For $\varepsilon=0.019$ this is illustrated in Figure~\ref{fig:graphs}b.
As the plot reveals, the graph is not as smooth as the strong stable fibres, it seems to be case~B2-i.
Even more, in this special example there exists a $T$-invariant measure $\nu$ with $\nu\{\varphi^- < -0.1\}=1$ and $h_\nu (T)>0$.
Indeed, $\dim _H\{\varphi^- < -0.1\}\geqslant 3/2$.
\footnote{Proof: 
Let $ C:=\left\{ \sum _{j\geqslant 1} a_j/4^j : a_j\in\{1,2\} \right\} $.
The level plot shows that $\sup _{y<-0.1}f_x\circ f_{\tau x}(y)<-0.1$ for all $x\in [1/4,3/4]$, which implies $\phi^-<-0.1$ on $ \tau^{-1} C $ in view of the pull-back construction of $\phi^-$.
As $\dim _H (C)=\frac{\log 2}{\log 4}=\frac{1}{2}$, we have $\dim _H\{\varphi^-<-0.1\}\geqslant\dim_H(\I\times (\tau^{-1}C))=3/2$.
Moreover, as $(C, \tau^2|_{C})$ is isomorphic to the full shift on two symbols (up to at most countably many points), there is a $T$-invariant measure on $\{\varphi^-<-0.1\}$ with entropy $\frac{1}{2}\log 2$.
}

\item If $0.038\leqslant\varepsilon\leqslant0.1$, the branches $f_{1/3}=f_{2/3}$ over these $2$-periodic points have a unique
negative fixed point.
%Together with the fact that 
As the values $\phi^+(x)$
and $\phi^-(x)$ are determined through $\tau(x)$, we have
$0,\,1/3,\,1/2,\,2/3,\,1\in P$, $\phi^-(0)=\phi^-(1/2)=\phi^-(1)> 0.6$
and $\phi^+(1/3)=\phi^+(2/3)<-0.6$. In particular, $0,\,1/2$ and $1$ are
interior points of the set $\{x\in\I:\phi^-(x)\geqslant 0.6\}$,
while $1/3$ and $2/3$ are interior points of $\{x\in\I:\phi^+(x)\leqslant -0.6\}$.
For $\varepsilon=0.04$ and $\varepsilon=0.08$ this is illustrated in Figures~\ref{fig:graphs}c and~\ref{fig:timeline}. For both parameters, transitions from negative to positive $y$-values and also from positive to negative $y$-values are possible. The difference between both parameters is that within $10^7$ iterations we observe just one transition from negative to positive values and no transition back for $\varepsilon=0.04$, whereas there are several such transitions in both directions for $\varepsilon=0.08$. We do not know whether these are B1- or B2- situations.
However, the lack of smoothness of the visible invariant graph and varying strong stable fibres seem to exclude B2-ii, so that $\nu(P)=1$ for any Gibbs measure $\nu$.
\end{compactenum}                   
\end{compactenum}
\begin{figure}
\begin{minipage}{0.95\textwidth}  
\includegraphics[width=0.5\textwidth,height=5cm,clip=true, trim=65mm 0mm 0mm 0mm]{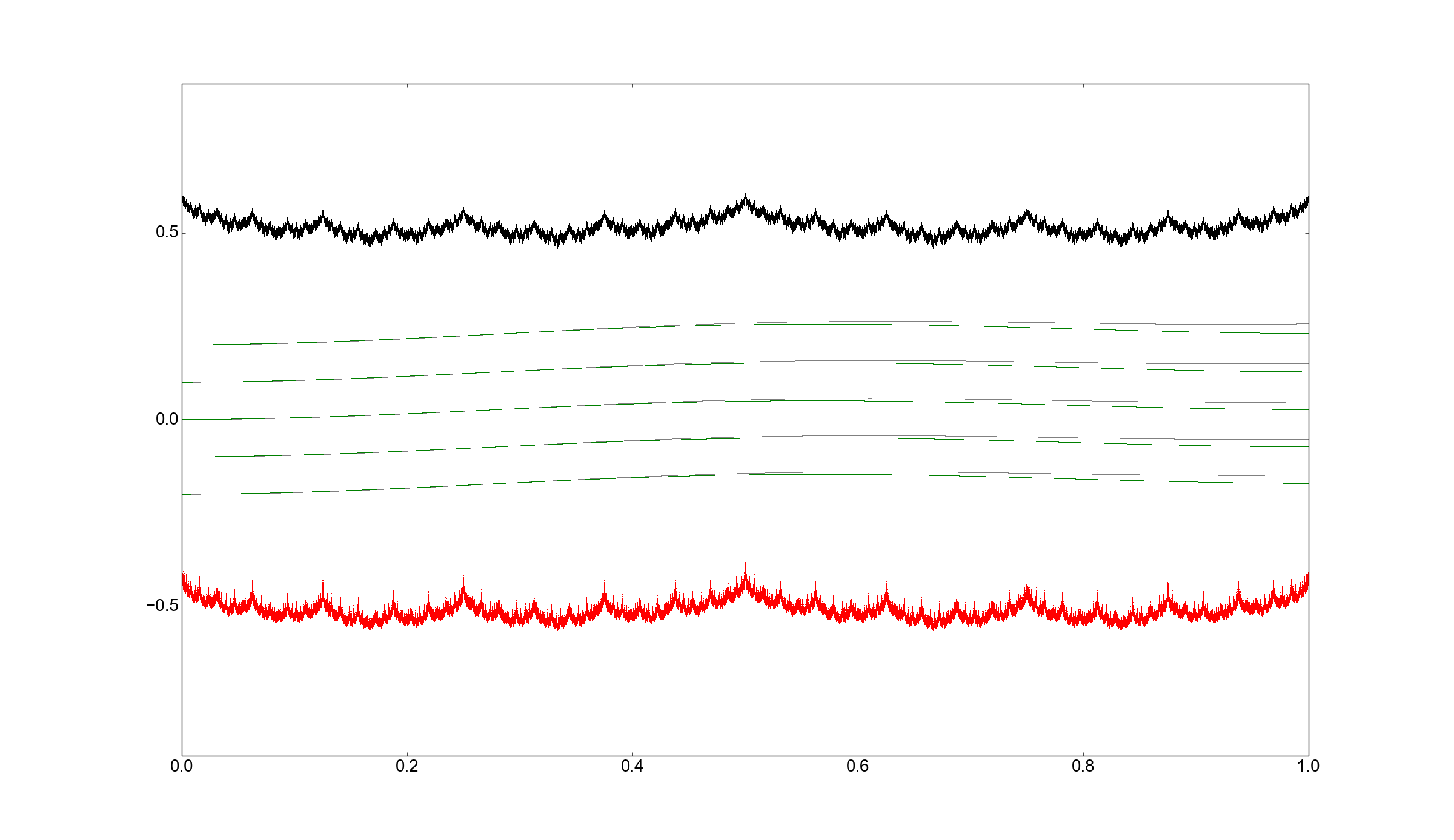}
\includegraphics[width=0.5\textwidth,height=5cm,clip=true, trim=0mm 10mm 0mm 3mm]{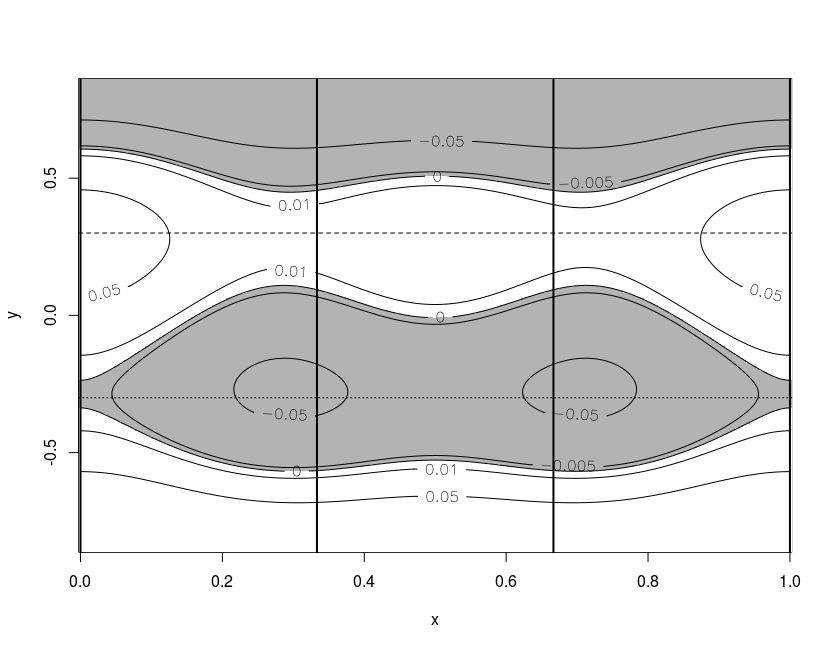}
\subcaption*{a) $\varepsilon =0.018$: The lhs plot shows two trajectories, one starting at $y=-1$ and one starting at $y=1$.
Each orbit starting above the dashed line (at $y_\varepsilon=0.3$) in the rhs plot will stay above it, and one starting below the dotted line (at $-y_\varepsilon=-0.3$) will stay below it. This is case~A.}
\end{minipage}
\begin{minipage}{0.95\textwidth}
\includegraphics[width=0.5\textwidth,height=5cm,clip=true, trim=65mm 0mm 0mm 0mm]{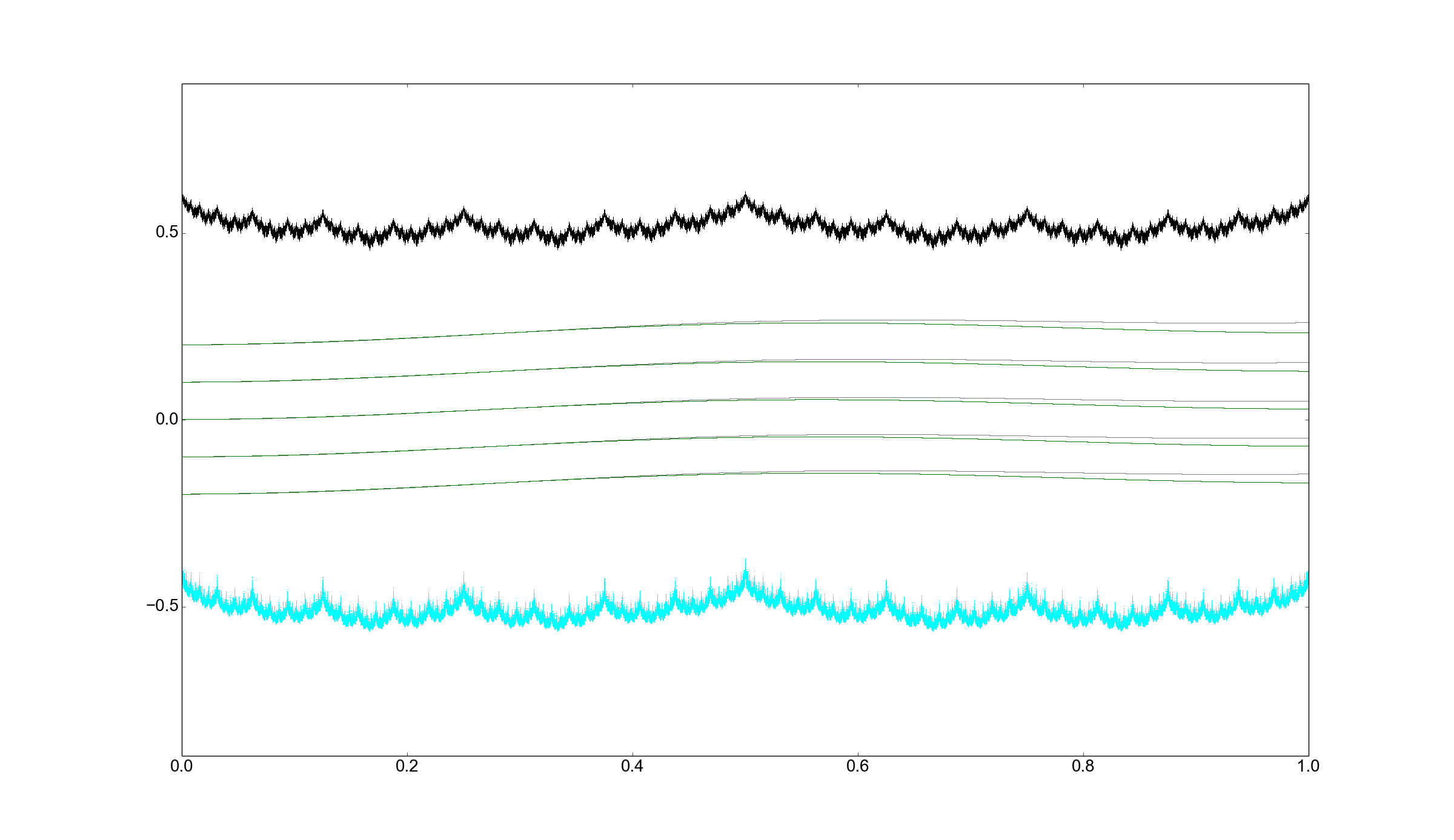}
\includegraphics[width=0.5\textwidth,height=5cm,clip=true, trim=0mm 10mm 0mm 3mm]{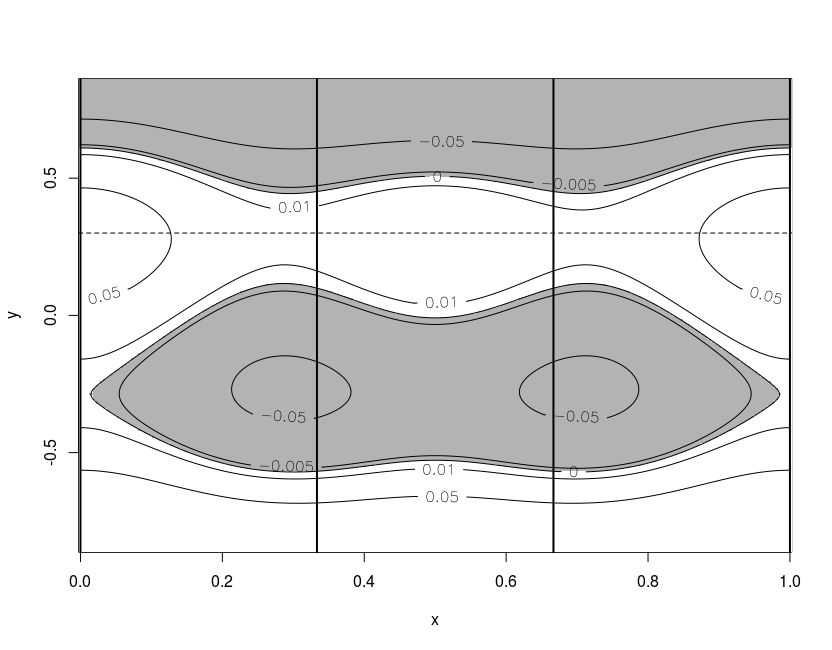}
\subcaption*{b) $\varepsilon =0.019$: The lhs plot shows two trajectories, one starting at $y=-1$ and one starting at $y=1$.
An orbit starting above the dashed line (at $y=0.3$) in the rhs plot will stay above this line forever.  As $\sup_{0\leqslant x\leqslant 1,|y|\geqslant0.3}(f_x\circ f_{\tau x})'(y)<0.99$, $\phi^+$ is continuous. The branch $f_0$ has a unique fixed point at $y_1 \approx 0.610$, whereas $f_{1/3}=f_{2/3}$ has three fixed points, the stable ones at $y_2 \approx -0.568$ and $y_3 \approx 0.451$. 
The orbits of all points whose $\xi$-coordinate is extremely close to $0$ approach $x=0$ and finally end up above the dashed line. The plotted orbit with negative $y$-values is transient -- after a much longer time it would cross to positive $y$-values. 
In view of the Corollary to theorem~\ref{theo:B-class}, this is case~B2-i.
}
\end{minipage}
\begin{minipage}{0.95\textwidth}
\includegraphics[width=0.5\textwidth,height=5cm,clip=true, trim=65mm 0mm 0mm 0mm]{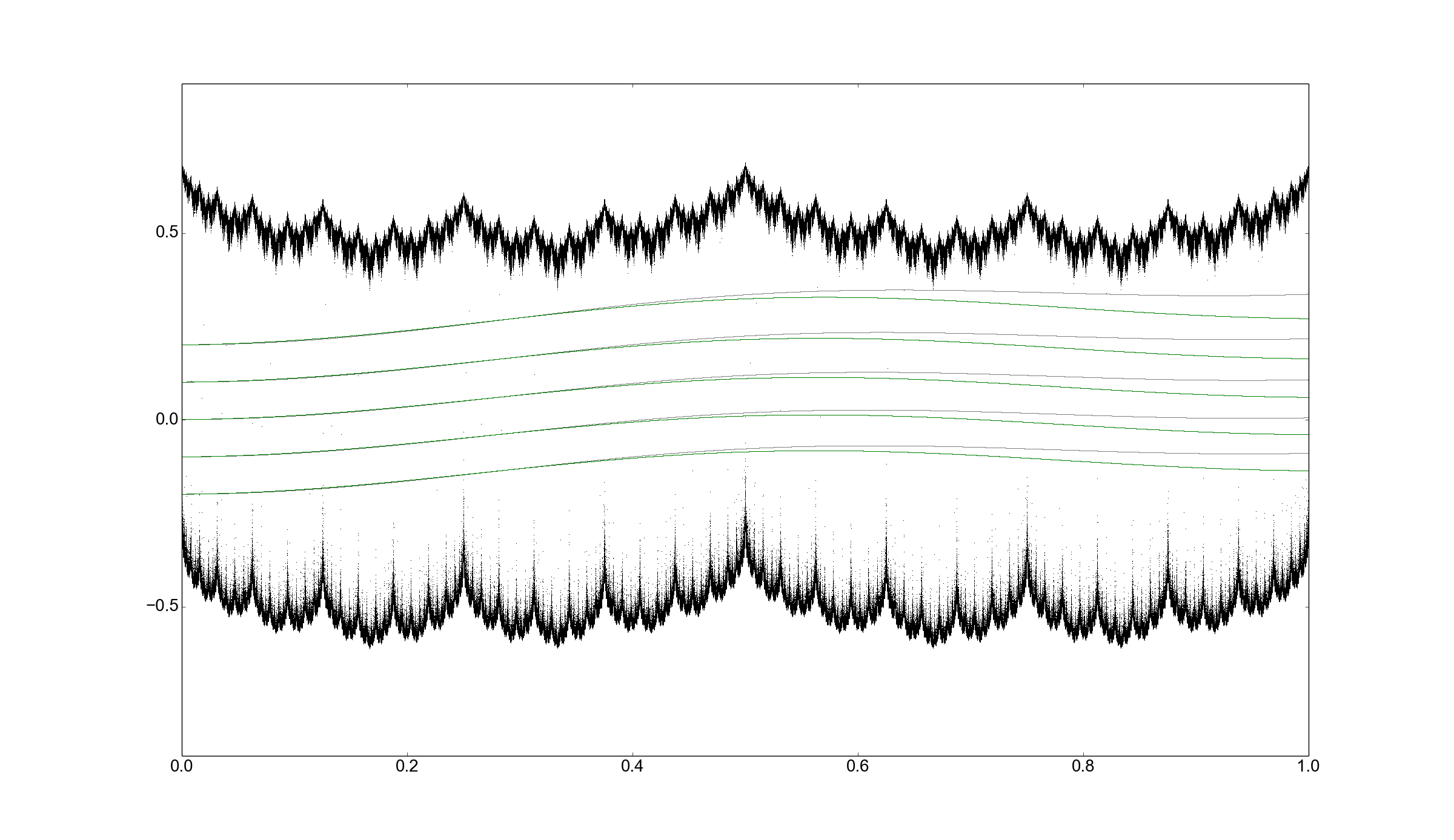}
\includegraphics[width=0.5\textwidth,height=5cm,clip=true, trim=0mm 10mm 0mm 3mm]{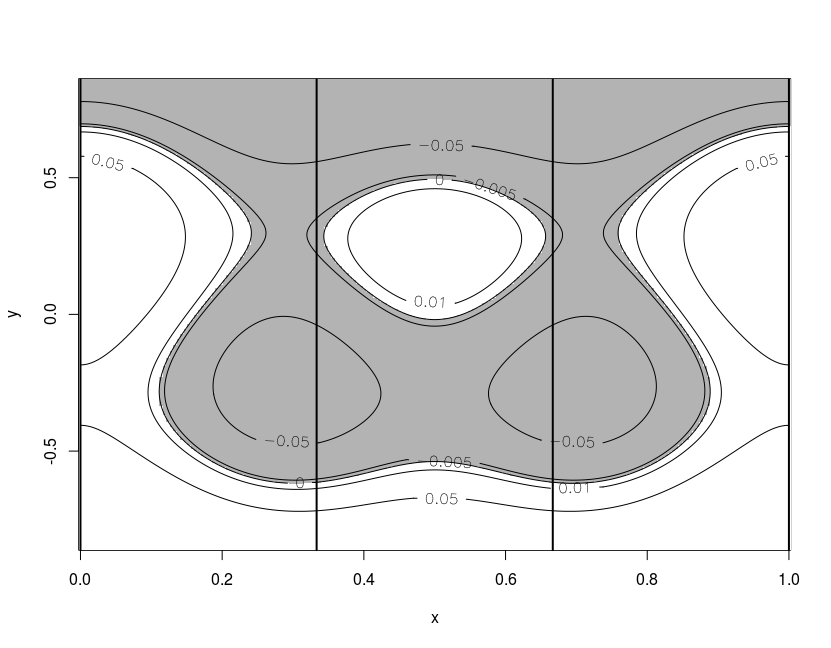}
\subcaption*{c) $\varepsilon =0.04$: 
The lhs plot shows a single trajectory that was started at $y=-1$, and which jumped to positive $y$-values after 972.858 iterations.
The branch $f_0$ has a unique fixed point at $y_1 \approx 0.687$ and $f_{1/3}=f_{2/3}$ have a unique fixed point at $y_2 \approx -0.614$. Hence $\phi^+(0)=\phi^-(0)=y_1>0$ and $\phi^-(1/3)=\phi^+(1/3)=y_2<0$. 
In particular $P\neq\emptyset$, so this is case B. The varying strong stable fibres exclude case~B2-ii, so that $m^2(P)=1$. Thus, this is either
case~B1 or~B2-i.}
\end{minipage}
\caption{The plots on the lhs represent trajectories of length $10^7$, from which short initial segments are discarded.
The smooth lines are 
numerically determined strong stable fibres
$\ellss_{(\xi,0.5,y)}$ for different values of $\xi$ and $y$. The gray regions in the level plots on the rhs consist of those points $(x,y)$, for which $f_x(f_{\tau(x)}(y))<y$. \label{fig:graphs}}
\end{figure}
\begin{figure}
\begin{minipage}{0.95\textwidth}
\includegraphics[width=0.5\textwidth,height=5cm,clip=true, trim=65mm 0mm 0mm 0mm]{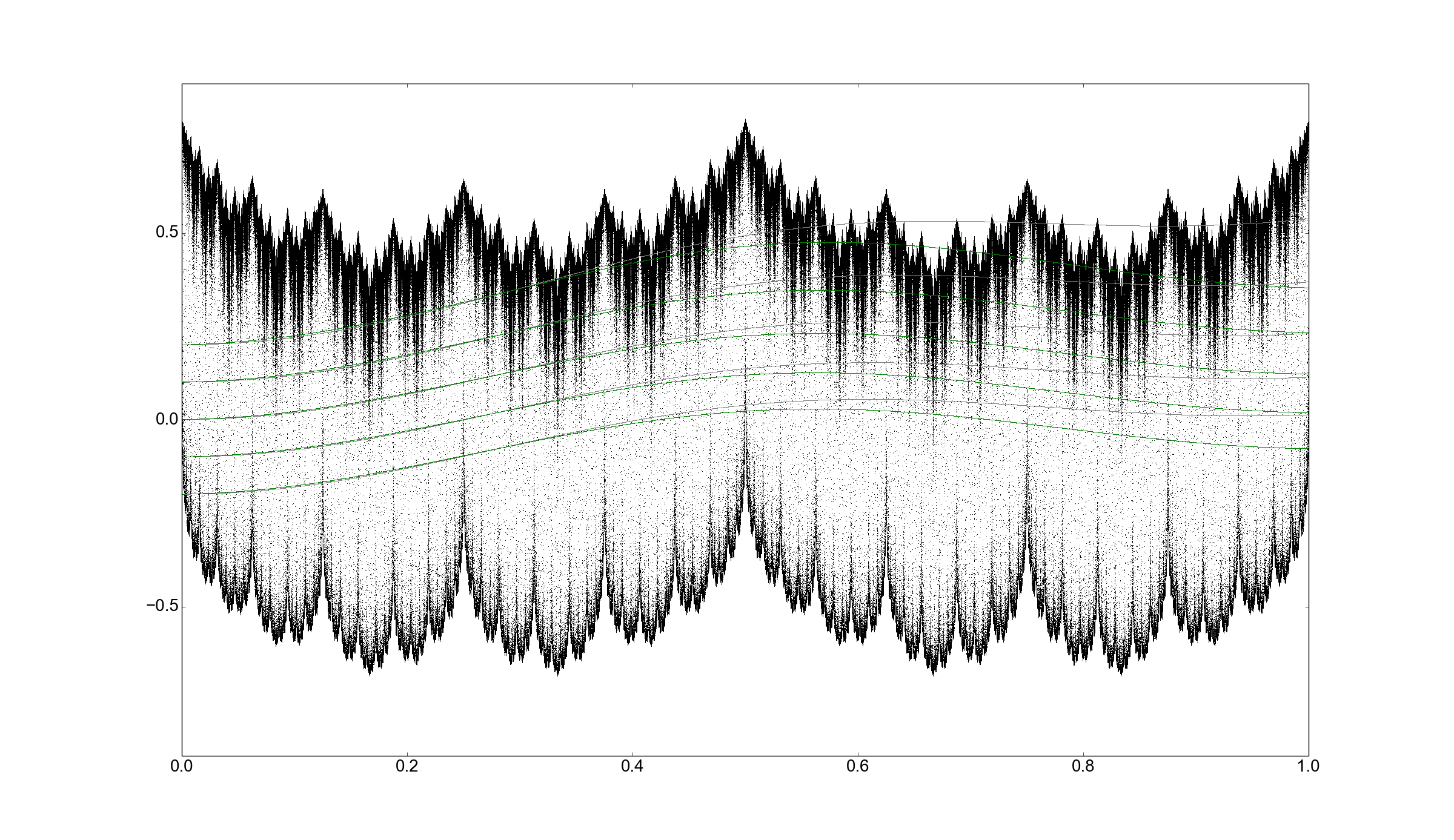}
\includegraphics[width=0.5\textwidth,height=5cm,clip=true, trim=45mm 0mm 50mm 0mm]{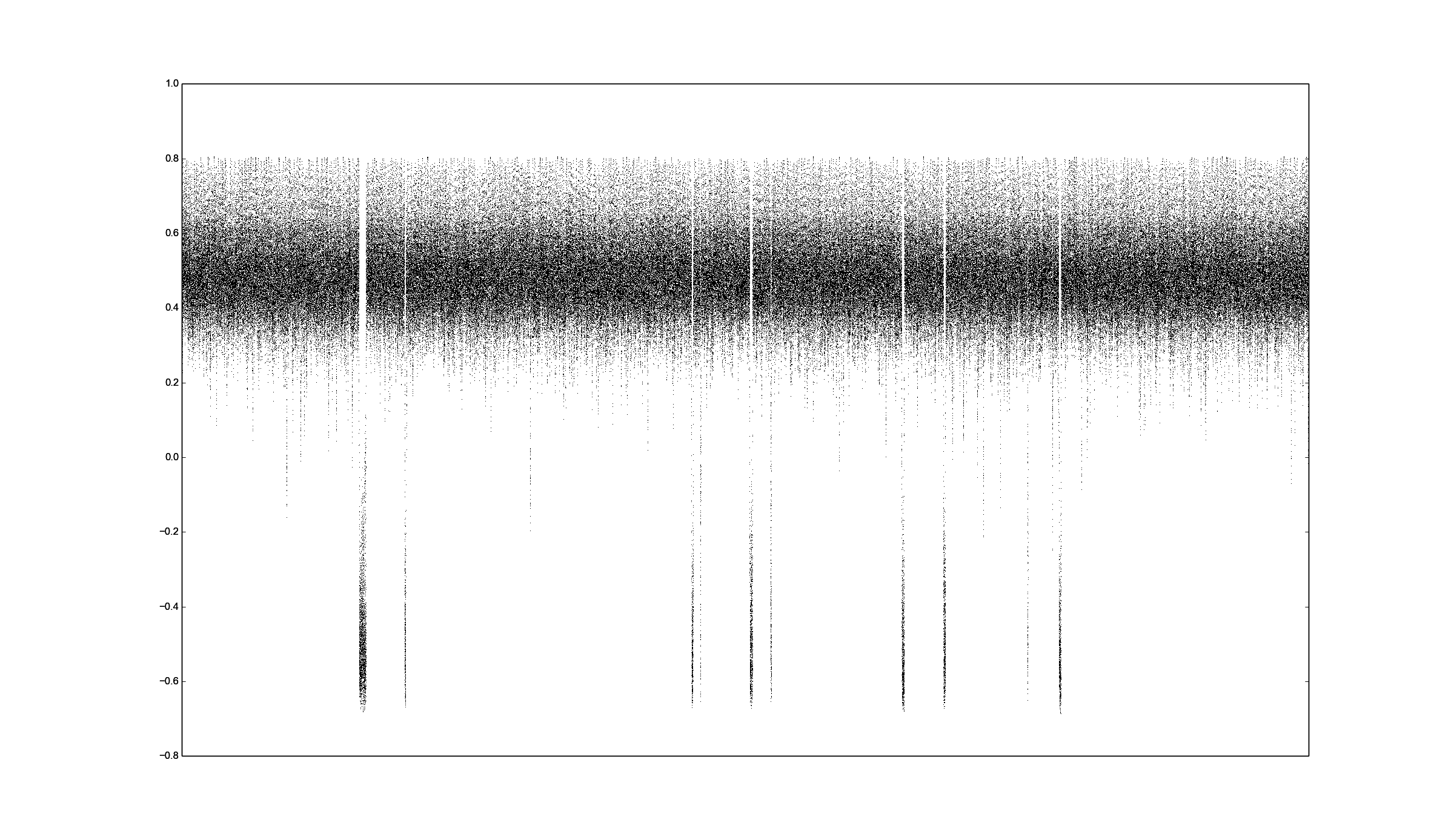}
\end{minipage}
\caption{$\varepsilon=0.08$: The situation is basically the same as for $\varepsilon=0.04$, but the funnels around the points $0$, $1/3$, $2/3$ and $1$, at which trajectories can cross from negative to postive $y$ values or vice versa, are broader, so that such crossings are observed more often. This is illustrated by the time series on the rhs. \label{fig:timeline}}
\end{figure}

\begin{remark}(Degenerate cases) \label{rem:degenerate} 
Consider any Gibbs measure $\nu\in\cE_T(\T^2)$, e.g. Lebesgue measure, and suppose $\nu(P)=0$. Cases A$_\nu2$) and
B2-ii)
are degenerate relative to $\nu$ in the sense that the graph of one of $\phi ^{+-}$, $\phi^{-+}$ and $\hat\phi$ coincides $\nu$-a.s. with some strong stable fibre $\gamma_\xi^*$, so that these fibres are identical for $\nu$-a.e.~$\xi$. In particular, there is an invariant graph of class $\CLip$, so that the dynamics are $\CLip$-conjugate to one where $f_x(0)=0$ for all $x$. Such systems are reasonably well understood, although some results have published proofs only for the more particular case of multiplicative forcing, see e.g. \cite{KeOt2012,Keller2014}. Theorem 2.2 of the thesis~
\cite{Otani2015}, however, applies with only minor modifications to case~B2 and allows to derive a variational formula for the Hausdorff dimension of the set $P=\{\varphi^-=\varphi^*=\varphi^+\}$ as in the subsequent theorem.
\end{remark}

\begin{theorem}\label{theo:dimension}
\begin{compactenum}[a)]
\item Suppose that $\inf_{\nu\in\cM_T(\T^2),\, \nu\, {\rm Gibbs}}\lambda^*(\nu)>0$. Then
case~A of Theorem~\ref{theo:top-class} applies, in particular $P=\emptyset$ and $\dimH(P)=0$.
\item
Suppose that $\inf_{\nu\in\cM_T(\T^2)}\lambda^*(\nu)<0$. Then case~B of Theorem~\ref{theo:top-class} applies, in particular $P$ is a dense $G_\delta$-set, and
\begin{equation}\label{eq:dim-formula}
\dimH\left(P\right)
=
1+\sup\left\{\frac{h_\nu(T)}{\int\log\tau'(x)\,d\nu(\xi,x)}: \nu\in\cM_T(\T^2), \lambda^*(\nu)\leqslant0\right\},
\end{equation}
where $h_\nu(T)$ denotes the Kolmogorov-Sinai entropy of $(T,\nu)$, which coincides with $h_{\nu^{(1)}}(\tau)$.
%$h_{\nu^{(2)}}(\tau)$.
The supremum in (\ref{eq:dim-formula}) is attained for a unique Gibbs measure.
\end{compactenum}
\end{theorem}

At present we have to leave open what values $\dimH(P)$ can attain when 
$\inf_{\nu\in\cM_T(\T^2)}\lambda^*(\nu)=\inf_{\nu\in\cM_T(\T^2),\, \nu\, {\rm Gibbs}}\lambda^*(\nu)=0$.

The proofs of all theorems are provided in Section~\ref{sec:main-proofs}.

\begin{remark}(Intermingled basins)
Nearly the same model as in this note was studied in \cite{Keller2015}, where the focus was on quantifying the phenomenon of intermingled basins using thermodynamic formalism. The formal difference to the present setting is that $f_x(I)=I$ is assumed in \cite{Keller2015} (so that the end points of $I$ describe constant invariant graphs given \emph{a priori}), while here we require $f_x(I)\subseteq I ^\circ$ (Hypothesis~\ref{hyp:weak-contraction}). 
Extending the fibres and fibre maps in the examples from \cite{Keller2015} slightly, puts us in the context of the present paper.
The result is that $\varphi^+$ from \cite{Keller2015} is the lower bounding graph of our $\widetilde{\cG\varphi^+}$, and similar for $\varphi^-$. 

Having this in mind, it is obvious that all examples from \cite{Keller2015} belong to our class A). Indeed, Example~2.3 and Figure~1 of \cite{Keller2015} illustrate situations, where all invariant measures lead to case~A1 (parameters $a\geqslant2$), and also 
examples, where Lebesgue measure leads to~A1, but where there are other invariant Gibbs measures\footnote{The measures $\nu^-$ and $\nu^+$ from \cite[Example 2.3]{Keller2015} with positive Lyapunov exponent on one of the constant invariant graphs
are equidistributions on periodic points, but in view of the variational principle there exist also Gibbs measures, for which the corresponding exponents are positive.} leading to case~A2 (parameters $a\leqslant1.2$).
\end{remark}

\section{Proofs of the main results}\label{sec:main-proofs}

\begin{proof}[Proof of Theorem~\ref{theo:top-class}]
As $\varphi^\pm(\xi,x)=\phi^\pm(x)$, 
the set $\cG\varphi^\pm$ can be written as $\T\times\cG\phi^\pm$  so that
$\overline{\cG\varphi^\pm}=\T\times\overline{\cG\phi^\pm}$. Therefore $\overline{\cG\varphi^\pm}$ is forward invariant under $f$, although the base map $T$ is not a homeomorphism: Indeed, let $(\xi,x,y)\in \overline{\cG\varphi^\pm}$. Then there are
$x_n$ $(n\in\N)$ such that $x_n\to x$ and $\varphi^\pm(\xi,x_n)\to y$. As $x'\mapsto T(\xi,x')$ is continuous for fixed $\xi$, it follows that
$f(\xi,x,y)=(T(\xi,x),f_{x}(y))=\lim_{n\to\infty}(T(\xi,x_n),f_{x_n}(\varphi^\pm(\xi,x_n))=\lim_{n\to\infty}(T(\xi,x_n),\varphi^\pm(T(\xi,x_n))\in\overline{\cG\varphi^\pm}$. The forward $f$-invariance of
$\widetilde{\cG\varphi^\pm}$ is an immediate consequence.

Let $\cD:=\widetilde{\cG\varphi^+}\cap\widetilde{\cG\varphi^-}$ and suppose that $\cD\neq\emptyset$. Then $\cD$ is a non-empty, closed, filled-in, forward $f$-invariant subset of $\T^2\times I$, so that $\pi(\cD)$ is a non-empty, closed, forward $T$-invariant subset of $\T^2$. As $\widetilde{\cG\varphi^\pm}=\T\times\widetilde{\cG\phi^\pm}$, the set $\pi(\cD)$ can be written as $\T\times D$. The $T$-forward invariance of this set implies it is dense in $\T^2$, and so $\pi(\cD)=\T^2$.

The inclusion $P\subseteq C^+\cap C^-$ follows at once from the semi-continuity of $\varphi^\pm$. For the reverse inclusion let $(\xi,x)\in C^+\cap C^-$. As $\pi(\cD)=\T^2$, there is $y\in I$ such that $(\xi,x,y)\in \cD$.
Hence there are $y^-,y^+\in I$, $y^+\leqslant y\leqslant y^-$, such that $(\xi,x,y^\pm)\in\overline{\cG\varphi^\pm}$.
 As $\varphi^\pm$ do not depend on $\xi$, there are $x_n^+,x_n^-\in\T$ such that
$x_n^\pm\to x$ and $\varphi^\pm(\xi,x_n^\pm)\to y^\pm$. As $(\xi,x)\in C^+\cap C^-$ we conclude that $\varphi^+(\xi,x)=\lim_{n\to\infty}\varphi^+(\xi,x_n^+)=y^+\leqslant y\leqslant y^-=\lim_{n\to\infty}\varphi^-(\xi,x_n^-)=\varphi^-(\xi,x)$, i.e. $(\xi,x)\in P$.

We turn to the assertions on the set $\cP=\overline{\cG\varphi^\pm|_P}$. As $\varphi^+,\varphi^-$ are invariant graphs, the set $P$ is forward $T$-invariant so that the graph $\cG\varphi^\pm|_P$ is forward $f$-invariant. From this, the forward $f$-invariance of $\cP$ follows as above by choosing the points $x_n$ such that $(\xi,x_n)\in P$. Finally, $\pi(\cP)=\T^2$ is proved in the same way as $\pi(\cD)=\T^2$ above.

It remains to prove that case~A is excluded if $\lambda^*(\nu)<0$ for some $\nu\in\cE_T(\T^2)$. Indeed, Proposition~\ref{prop:classification-general} shows that 
$\varphi^+=\varphi^*=\varphi^-$ $\nu$-a.e., and this is incompatible with A.
\end{proof}

For the proofs of the further theorems we need several preparatory lemmas.
The first one states formally the well-known $f$-invariance of the family of strong stable fibres. We skip its proof.
\begin{lemma}	\label{lem:lss_invariant}
%We have
\[
	f _{(\xi,u)} ^n \left( \ellss _{(\xi, x, y)} (u) \right) = \ellss _{f ^n (\xi, x, y)} \left( \pi _2 \circ T ^n (\xi, u ) \right)
\]
for all $ u \in \dom _{(\xi, x, y)} $, $ (\xi, x, y) \in \I ^2 \times I $ and $ n \in \N$, where $ \pi _2 $ is the projection on the second coordinate.
\end{lemma}

\begin{lemma}	\label{lem:lss_domain}
There exists a constant $ \delta > 0 $ such that $ ( x - \delta , x + \delta ) \cap \I  \subseteq \dom _{(\xi, x, y)} $ for all $ (\xi, x, y) \in  \I  ^2 \times I $.
\end{lemma}

\begin{proof}
Let $ I = [s, t] $ and $J = [s',t']$ be the intervals of Hypothesis \ref{hyp:weak-contraction}.
By Lemma \ref{lemma:shortX3},
% the function $ X _3 $ is well-defined and bounded on $ \T ^2 \times J $.
%Hence
\[
	L := \sup _{(\xi, x,y) \in \I ^2 \times J} | X _3 (\xi, x, y) | < \infty .
\]
Let $ \varepsilon _0 := \min\{ s-s', t'-t\} > 0 $ and consider a triple
$ ( \xi, x , y ) \in \I  ^2 \times I $.
By the mean value theorem, for $ u \in \dom _{(\xi, x, y)} $, we have
\[
	 \ellss _{(\xi, x, y)} (u) \leqslant y + 	 \sup _{\tilde{u} \in \dom _{(\xi, x, y)}} 	 \left| ( \ellss _{(\xi, x, y)} ) ' (\tilde{u})  \right| \, | u - x |  \leqslant t + L | u - x | ,
\]
and similarly $ \ellss _{(\xi, x, y)} (u) \geqslant s - L | u - x | $.
Let $ \delta := \varepsilon _0 / L $. Then
\[
	\ellss _{(\xi, x, y)} (u) \in J
\]
for all $ u \in \dom _{(\xi, x, y)} \cap ( x - \delta  , x + \delta  ) $.
Hence, the Picard-Lindelöf theorem guarantees that $\ellss_{(\xi,x,y)}(u)$, the solution of the initial value problem (\ref{eq:IVP}) extends at least to the interval $(x-\delta,x+\delta)$,
i.e. $ ( x - \delta , x + \delta ) \subseteq \dom _{(\xi, x, y)}  $.
\end{proof}

\begin{lemma}\label{lem:lss_lyapunov}
$\lambda _{(\xi, x, y)} = \lambda _{(\xi, u, \ellss _{(\xi, x, y)} (u) )}$
for all $ u \in \dom _{(\xi, x, y)} $ and $ (\xi, x, y) \in \I ^2 \times I $.
In other words: $\lambda_{(\xi,x,y)}$ is constant along strong stable fibres.
\end{lemma}

\begin{proof}
Let $ (\xi, x, y ) \in \I  ^2 \times I $ and $ u \in \dom _{(\xi, x, y )} $.
By Lemma \ref{lem:lss_invariant}, we have
\[
	f _{(\xi,u)} ^n \left( \ellss _{(\xi, x, y)} (u) \right) = \ellss _{f ^n (\xi, x, y)} \left( \pi _2 \circ T ^n (\xi, u ) \right)
\]
for each $ n $, where $ \pi _2 $ is the projection on the second coordinate.
Observe also that $ f ^n (\xi, x, y) \in \I ^2 \times J $.
Then $\left| (  \ellss _{f ^n (\xi, x, y)} ) ' \left( \pi _2 \circ T ^n (\xi, u ) \right) \right| \leqslant L  $, so that
\begin{eqnarray*}
	\left| f _{(\xi,u)} ^n \left( \ellss _{(\xi, x, y)} (u) \right) - f _{(\xi,x)} ^n ( y ) \right|	 &  = &\left| f _{(\xi,u)} ^n \left( \ellss _{(\xi, x, y)} (u) \right) -  f _{(\xi,x)} ^n \left( \ellss _{(\xi, x, y)} (x) \right) \right| \\
	 &  = &\left|  \ellss _{f ^n (\xi, x, y)} \left( \pi _2 \circ T ^n (\xi, u ) \right) - \ellss _{f ^n (\xi, x, y)} \left( \pi _2 \circ T ^n (\xi, x ) \right)   \right| \\
		& \leqslant &	L \;  \left| \pi _2 \circ T ^n (\xi, x ) - \pi _2 \circ T ^n (\xi, u ) \right|		\stackrel{n \rightarrow \infty}{\longrightarrow} 0  .
\end{eqnarray*}
Observing the continuity of $(x,y)\mapsto\log f_{x}'(y)$, the claim follows from
\begin{equation*}
\log (f_{(\xi,x)}^n)'(y)
=\sum_{j=0}^{n-1}f'_{\pi_2\circ T^j(\xi,x)}(f_{(\xi,x)}^j(y))\ .
\end{equation*}
\end{proof}

\begin{lemma}	\label{lem:private_discussion}
Let $ \nu \in \cM _\tau ( \T ^2 ) $ be a Gibbs measure with $\nu(P)=0$.
Then  $ \dom _{(\theta, \varphi ^\ast  (\theta)) } = \I $ for $ \nu  $-a.a.~$ \theta $,
\begin{equation}\label{eq:product-measure}
	( \nu  \otimes \nu ^{(2)} ) \left\{ ( (\xi ,x ), u) : \varphi ^\ast ( \xi , u) = \ell ^{ss} _{(\xi, x, \varphi ^\ast (\xi, x ))} ( u )  \right\} = 1\,,
\end{equation}
and $\phi ^-\leqslant\ellss _{(\theta, \varphi ^\ast (\theta) )}\leqslant\phi ^+$ for $\nu$-a.a. $\theta$.
\end{lemma}

\begin{proof}
As $\nu(P)=0$, we are in case~(ii) or (iii) of Proposition~\ref{prop:classification-general}.
W.l.o.g. let $\lambda^-(\nu)<0$, since otherwise $\lambda^+(\nu)<0$ so that the proof is similar.
By Corollary \ref{coro:exponent}, there is a set $A\subseteq\I^2$ with $\nu(A)=1$ such that
\begin{equation}\label{eq:phi_ell_proof0}
	\lambda _{(\theta, y) } =
	\begin{cases}
		\lambda ^\ast(\nu)\geqslant0 & \mbox{for } y = \varphi ^\ast ( \theta ) \\
		< 0 & \mbox{for } y <\varphi^\ast(\theta)
	\end{cases}
\end{equation}
for all $\theta\in A$. (The value for $y>\varphi^\ast(\theta)$ does not matter.)
Moreover, by Lemma \ref{lem:lss_lyapunov},
\begin{equation*}
	\lambda _{((\xi,u), \ellss _{(\theta, y)} (u) ) } = \lambda _{(\theta, y ) }
\end{equation*}
for all $\theta=(\xi,x)\in\I^2$, $y\in I$ and $ u \in \dom _{(\theta, y)} $. Hence
\begin{equation} \label{eq:phi_ell_proof1}
\lambda _{((\xi,u), \ellss _{(\theta, y)} (u) ) } 
=
	\begin{cases}
		\lambda ^\ast(\nu)\geqslant0 & \mbox{if } y = \varphi ^\ast ( \theta ) \\
		< 0 & \mbox{if } y <\varphi ^\ast ( \theta ) 
	\end{cases} 	
\end{equation}
for all $\theta=(\xi,x)\in A$ and $ u \in \dom _{(\theta, y)} $.
A combination of (\ref{eq:phi_ell_proof0}) and (\ref{eq:phi_ell_proof1}) implies
\begin{equation*}
\tilde y=\varphi^*(\theta)
\quad\Rightarrow\quad
\lambda _{((\xi,u), \ellss _{(\theta, \tilde y ) } (u))}  =\lambda^*(\nu)
\quad\Rightarrow\quad
\ellss _{(\theta,\tilde y ) } (u)\geqslant\varphi^*(\xi,u)
\end{equation*}
and
\begin{equation*}
\tilde y<\varphi^*(\theta)
\quad\Leftrightarrow\quad
\forall y\leqslant\tilde{ y}:\
\lambda _{((\xi,u), \ellss _{(\theta, y ) } (u))}  <0
\quad\Leftrightarrow\quad
\forall y\leqslant\tilde{ y}:\
\ellss _{(\theta, y ) } (u)<\varphi^*(\xi,u)
\end{equation*}
for all $\theta=(\xi,x)\in A$, $\tilde y\in I$ and $u\in(x-\delta,x+\delta)\cap\I\cap A_\xi\subseteq\dom _{(\theta, y)}\cap A_\xi$, where $ A _\xi := \{ u \in \I : ( \xi, u ) \in A \} $. As $y\mapsto\ellss_{(\theta,y)}(u)$ is strictly increasing and continuous, this shows that 
\begin{equation}
	\ellss _{(\theta, \varphi ^\ast ( \theta ))} (u) = \varphi ^\ast ( \xi, u) \in I	\label{eq:phi_ell_proof}
\end{equation}
for all $\theta= ( \xi, x ) \in A $ and 
$u\in(x-\delta,x+\delta)\cap\I\cap A_\xi$.

As $ \nu ( A ) = 1 = ( \nu ^{(1)} \otimes \nu^{(2)})  ( A ) $,
Fubini's theorem implies $ \nu^{(2)} ( A _\xi ) = 1 $ for all $ \xi \in R $ for some $ \nu ^{(1)} $-full set $ R \subseteq \I $. As $\nu^{(2)}$ has full support, all these $A_\xi$ are dense in $\I$. As 
$\ellss _{(\theta, \varphi ^\ast ( \theta ))}$ is continuous,
it now follows from \eqref{eq:phi_ell_proof}
that $\ellss _{(\xi,x, \varphi ^\ast ( \xi,x ))}(u)\in I$ for all $\xi\in R$,  $x\in A_\xi$ and all $u\in(x-\delta,x+\delta)\cap\I$.
As $\dom_{(\theta,\varphi^*(\theta))}=\dom_{(\xi,u,\ellss_{(\theta,\varphi^*(\theta))}(u))}$ for all $\theta=(\xi,x)\in\I^2$ and  $u\in 
(x-\delta,x+\delta)\cap\I\subseteq\dom_{(\theta,\varphi^*(\theta))}$ by definition of this interval, a repeated application of
Lemma~\ref{lem:lss_domain} implies $\dom_{(\theta,\varphi^*(\theta))}=\I$ for   all  $\theta=(\xi,x)$ from a set $\tilde{A}$ with $\nu(\tilde{A})=1$.
Hence
\[
	\left\{ ((\xi, x) , u ) : (\xi, x ) \in \tilde{A}  \mbox{ and }  u \in A_\xi \right\} \subseteq    \left\{ ( (\xi ,x ), u) : \varphi ^\ast ( \xi , u) = \ell ^{ss} _{(\xi, x, \varphi ^\ast (\xi, x ))} ( u )  \right\} .
\]
As the l.h.s. set has full $  \nu \otimes \nu^{(2)} $-measure, this finishes the proof of (\ref{eq:product-measure}).

An immediate consequence of (\ref{eq:product-measure}) is
\[
	( \nu \otimes \nu^{(2)} ) \left\{ ( \theta, u) : \phi ^- (u)\leqslant \ell ^{ss} _{(\theta, \varphi ^\ast (\theta))} ( u )\leqslant\phi ^+ (u)\right\} = 1 \,.
\]
In view of the semi-continuity of the bounding graphs, the continuity of the strong stable fibres and the fact $\supp (\nu ^{(2)}) = \T$, this finally implies
$\phi ^-\leqslant\ellss _{(\theta, \varphi ^\ast (\theta) )}\leqslant\phi ^+$ for $\nu$-a.a. $\theta$.
\end{proof}

\begin{proof}[Proof of Theorem~\ref{theo:A-class}]
Recall from Theorem~\ref{theo:top-class} that the general assumption in case~A is
$\widetilde{\cG\varphi^+}\cap\widetilde{\cG\varphi^-}=\emptyset$.
It is a consequence of Proposition~\ref{prop:classification-general} that either
$\nu_{\varphi^*}(\widetilde{\cG\varphi^+}\cup\widetilde{\cG\varphi^-})=0$ (case~A$_\nu1$) or
$\nu_{\varphi^*}(\widetilde{\cG\varphi^+})=1$ (case~A$_\nu2^+$)
or $\nu_{\varphi^*}(\widetilde{\cG\varphi^-})=1$ (case~A$_\nu2^-$). In all these cases, the assertions on the Lyapunov exponents follow from the same proposition.
So we only must prove the additional assertions.

Recall that $P=\emptyset$ in case~A of Theorem~\ref{theo:top-class}, in particular $\nu(P)=0$.
For $\xi\in\I$ let $\varphi_\xi^*:\I\to\R, u\mapsto\varphi^*(\xi,u)$.
Then $(\nu_{\varphi^*})_\xi=\nu_\xi\circ(\varphi_\xi^*)^{-1}
\approx\nu^{(2)}\circ(\varphi_\xi^*)^{-1}$ for $\nu^{(1)}$-a.a. $\xi$.
As $\varphi_\xi^*=\ellss_{(\xi,x,\varphi^*(\xi,x))}$ $\nu^{(2)}$-a.s. for $\nu$-a.a. $(\xi,x)$ by
Lemma~\ref{lem:private_discussion}, this shows that, for $\nu^{(1)}$-a.a. $\xi$ and $\nu^{(2)}$-a.a. $x$,
\begin{equation}\label{eq:gamma-ellss}
\text{the measure $(\nu_{\varphi^*})_\xi$ is supported by the graph of $\ellss_{(\xi,x,\varphi^*(\xi,x))}$.
}
\end{equation}
In particular, given $\xi$, this is the same graph for $\nu^{(2)}$-a.a. $x$, and we denote it by $\gamma_\xi^*$.

\begin{compactenum}[i)]
\item[A$_\nu1$)] 
Let
\[
A=\left\{(\xi,x)\in\I^2:\ \exists u\in\I\text{ s.t. }
(\xi,u,\ellss_{(\xi,x,\varphi^*(\xi,x))}(u))\in\widetilde{\cG\varphi^+}\right\}.
\]
$A$ is Borel-measurable, because it is the intersection of the sets $V_n$ $(n\in\N)$ of all $(\xi,x)$, for which the graph of the continuous function 
$u\mapsto \ellss_{(\xi,x,\varphi^*(\xi,x))}(u)$ has distance at most $1/n$ from $\widetilde{\cG\phi^+}$, and membership in a such set $V_n$ is determined by the values of 
$\ellss_{(\xi,x,\varphi^*(\xi,x))}(u)$ on any 
fixed countable dense set of arguments $u\in\T$.

Suppose for a contradiction that $\nu(A)>0$. Let $A_\infty:=\limsup_{n\to\infty}T^{n}(A)$. Then $\nu(A_\infty)=1$ by ergodicity. Let $(\xi,x)\in A_\infty$. There are $n_1<n_2<\dots$ such that $(\xi_k,x_k):=T^{-n_k}(\xi,x)\in A$. Hence there are $u_1,u_2,\dots\in\I$ such that 
$(\xi_k,u_k,\ellss_{(\xi_k,x_k,\varphi^*(\xi_k,x_k))}(u_k))\in\widetilde{\cG\varphi^+}$ for all $k\in\N$. 
Let $u_k'=\pi_2\circ T^{n_k}(\xi_k,u_k)$.Then $|u_k'-x|\leqslant\max\{a,1-a\}^{n_k}$, and
\[
\qquad
(\xi,u_k',\ellss_{(\xi,x,\varphi^*(\xi,x))}(u_k'))
=
f^{n_k}\left((\xi_k,u_k,\ellss_{(\xi_k,x_k,\varphi^*(\xi_k,x_k))}(u_k))\right)
\in\widetilde{\cG\varphi^+}\quad\text{for all }k\in\N\ ,
\]
where we used Lemma \ref{lem:lss_invariant} and the $f$-invariance of $\widetilde{\cG\varphi^+}$.
As $u\mapsto\ellss_{(\xi,x,\varphi^*(\xi,x))}(u)$ is continuous and $\lim_{k\to\infty}u_k'= x$, we conclude that 
$(\xi,x,\varphi^*(\xi,x))=
(\xi,x,\ellss_{(\xi,x,\varphi^*(\xi,x))}(x))
\in\widetilde{\cG\varphi^+}$.
Recall that this holds for all $(\xi,x)\in A_\infty$ and that $\nu(A_\infty)=1$, so that $\nu_{\varphi^*}(\widetilde{\cG\varphi^+})=1$, which is excluded in case~A$_\nu1$.

We thus proved that $\gamma_\xi^*\cap\widetilde{\cG\varphi^+}=\emptyset$ for $\nu^{(1)}$-a.a. $\xi$.
The proof for $\widetilde{\cG\varphi^-}$ is the same.

\item[A$_\nu2^+$)] The lower bounding graph
$\varphi^{+-}(\theta):=\min\{y\in I: (\theta,y)\in\widetilde{\cG\varphi^+}\}$
 of $\widetilde{\cG\varphi^+}$ is invariant and lower semi-continuous
 by construction. Proposition~\ref{prop:classification-general} implies that $\varphi^*=\varphi^{+-}$ $\nu$-almost surely.
  As $\varphi^+$ does not depend on $\xi$, the same holds for $\varphi^{+-}$, i.e $\varphi^{+-}(\xi,x)=\phi^{+-}(x)$.

For $\xi\in \I$, let $V_\xi=\{x\in\I: \varphi^*(\xi,x)=\varphi^{+-}(\xi,x)\}$, and
denote by $U$ the set of
those $\xi\in\I$ for which $\nu^{(2)}(V_\xi)=1$ and for which (\ref{eq:gamma-ellss}) holds for $\nu^{(2)}$-a.a. $x$. Then $\nu^{(1)}(U)=1$, and $V_\xi$ is dense in $\I$ for all $\xi\in U$.

Let $\xi\in U$. Then
\begin{equation*}
\begin{split}
\psi_\xi(u)
&=
\lim_{V_\xi\ni u'\to u}\psi_\xi(u')
=
\lim_{V_\xi\ni u'\to u}\varphi^*(\xi,u')
\leqslant
\liminf_{V_\xi\ni u'\to u}\varphi^+(\xi,u')
=
\liminf_{V_\xi\ni u'\to u}\phi^+(u')\\
&\leqslant
\phi^+(u)\quad\text{for all }u\in \I\,,
\end{split}
\end{equation*}
where we used the upper semi-continuity of $\phi^+$ in the last step. 
As $\psi_\xi$ is continuous, this proves that $\psi_\xi\leqslant\phi^{+-}$.

In order to prove the reverse inequality,
denote 
$\psi_\xi:=\ellss_{(\xi,x,\varphi^*(\xi,x))}$ for the moment. 
Then $\phi^*(\xi,u)=\psi_\xi(u)$ for $\nu^{(2)}$-a.a. $u$.
Hence $(\xi,u,\psi_\xi(u))\in\widetilde{\cG\varphi^+}$ for $\nu^{(2)}$-a.a. $u$, in particular for a dense set of $u\in I$. As $\psi_\xi$ is continuous and 
$\widetilde{\cG\varphi^+}$ is closed, it follows that
$(\xi,u,\psi_\xi(u))\in\widetilde{\cG\varphi^+}$ for all $u\in\I$, equivalently that $\psi_\xi\geqslant\phi^{+-}$. 

Finally, observe that there are still the two possibilities: a) $\varphi^{+-}<\varphi^+$ $\nu$-a.s. or b) $\varphi^{+-}=\varphi^+$ $\nu$-a.s.

\item[A$_\nu2^-$)] This case is treated in complete analogy to the foregoing one.
\end{compactenum}
\end{proof}

\begin{proof}[Proof of Theorem~\ref{theo:B-class}]
Now the case where $P=C^+\cap C^-\neq\emptyset$ is considered, so that $P$ is indeed a dense $G_\delta$-set by Theorem~\ref{theo:top-class}.
We begin the proof with the following observation:
Let $\nu\in\cE_T(\I^2)$ be any Gibbs measure with $ \nu ( P ) = 0 $.
From Lemma~\ref{lem:private_discussion} we know that $\phi ^-\leqslant\ellss _{(\theta, \varphi ^\ast (\theta) )}\leqslant\phi ^+$ for some $\theta\in\T ^2$, so the three functions coincide on the set $P$.
As the function $(\tilde\xi,\tilde x)\mapsto\ellss _{(\theta, \varphi ^\ast (\theta) )}(\tilde x)$ is continuous, its graph is exactly $\cP$.
%In particular, for $\nu$-a.a.~$\theta=(\xi,x)$, all three functions coincide on the residual set $P_\xi$ (which is the same for all $\xi\in\I$), and as all $u\mapsto\ellss _{(\theta, \varphi ^\ast (\theta) )}(u)$ are continuous functions on $\I$,
%the compact, $f$-invariant set $\cP$ is the graph of a continuous function, call it $\hat\varphi:\I^2\to\I$.
In particular, case~B1 is excluded.

Hence, in case~B1 only Gibbs measures $\nu$ with $\nu(P)=1$ (and hence with $\lambda^*(\nu)\leqslant0$, see Proposition~\ref{prop:classification-general}) are possible.

We turn to case~B2, i.e. we assume that
$\#\{y\in I:(\theta,y)\in\cP\}=1$ for all $\theta\in\T^2$.
Then there is a function $\hat\varphi :\T^2\to I$ such that $\cP$ is the graph of $\hat\varphi$.
Observe that $\hat\varphi$ is a $f$-invariant continuous function, as $\cP$ is a $f$-invariant compact set, and that $\hat\varphi(\xi,x)=\hat\phi(x)$, as the set $\cP$ is defined in terms of the functions $\varphi^\pm(\xi,x)=\phi^\pm(x)$.

Let $\nu$ be any Gibbs measure.\\
- If $\nu(P)=1$, we are in case~(i) of Proposition~\ref{prop:classification-general}.
Hence $\hat\varphi = \varphi ^\ast$ $\nu$-a.s. and $\lambda^*(\nu)\leqslant0$.
\\ 
- Otherwise $\nu(P)=0$, because $P$ is $T$-invariant and $\nu$ is ergodic.
%Hence $\nu\{\varphi^+>\varphi^-\}=1$, and we are in case (ii) or (iii) of Proposition~\ref{prop:classification-general}, so that $\lambda^*(\nu)\geqslant0$.
By Lemma \ref{lem:private_discussion} we have $\phi ^-\leqslant \ellss _{(\theta, \varphi ^\ast (\theta))}\leqslant\phi ^+$ for $\nu$-a.a.~$\theta$, which implies $\hat\phi =\ellss _{(\theta, \varphi ^\ast (\theta))}$ for $\nu$-a.a. $\theta$.
Inserting this into \eqref{eq:product-measure}, $(\nu ^{(1)}\times\nu^{(2)})\left\{(\xi,u) : \varphi ^\ast (\xi, u)=\hat\phi(u)\right\}=1$, i.e. $\hat\varphi=\varphi ^\ast $ $\nu$-a.s.

It remains to consider the regularity of $\hat\phi$. So assume that $\hat\phi$  is everywhere differentiable with bounded derivative. As $\hat\varphi(\xi,x)=\hat\phi(x)$ and $\hat\varphi(T(\xi,x))=f_x(\hat\varphi(\xi,x))$, we have
\begin{equation*}
\frac{\partial\hat\varphi}{\partial x}(T(\xi,x))\cdot\sigma(\xi)
=
\frac{\partial f_x}{\partial x}(\hat\varphi(\xi,x))
+
f_x'(\hat\varphi(\xi,x))\cdot \frac{\partial\hat\varphi}{\partial x}(\xi,x)\ ,
\end{equation*}
i.e.
\begin{equation*}
\begin{split}
\frac{\partial\hat\varphi}{\partial x}(\xi,x)
&=
-\frac{\frac{\partial f_x}{\partial x}(\hat\varphi(\xi,x))}{f_x'(\hat\varphi(\xi,x))}   
+
\frac{\sigma(\xi)}{f_x'(\hat\varphi(\xi,x))}\cdot \frac{\partial\hat\varphi}{\partial x}(T(\xi,x))\\
&=
-A(\xi,x,\hat\varphi(\xi,x))
+\Gamma(\xi,x,\hat\varphi(\xi,x))\cdot \frac{\partial\hat\varphi}{\partial x}(T(\xi,x))
\end{split}
\end{equation*}
with $A$ and $\Gamma$ as defined in the proof of Lemma~\ref{lemma:shortX3}.
Combining this identity with equation \eqref{eq:equivariance3} from the same proof we obtain
for the difference $H(\xi,x):=\frac{\partial\hat\varphi}{\partial x}(\xi,x)-X_3(\xi,x,\hat\varphi(\xi,x))$,
\begin{equation*}
\begin{split}
H(\xi,x)
&=
\Gamma(\xi,x,\hat\varphi(\xi,x))\cdot\left(
\frac{\partial\hat\varphi}{\partial x}(T(\xi,x))-X_3(f(\xi,x,\hat\varphi(\xi,x)))\right)\\
&=
\Gamma(\xi,x,\hat\varphi(\xi,x))\cdot H(T(\xi,x))\ ,
\end{split}
\end{equation*}
where we used $f(\xi,x,\hat\varphi(\xi,x))=(T(\xi,x),f_x(\hat\varphi(\xi,x)))=(T(\xi,x),\hat\varphi(T(\xi,x)))$ for the last equality.
Since $ \| \Gamma \| _{C ( \T ^2 \times J )} < 1 $ by Hypotheses \ref{hyp:weak-contraction},
\begin{equation*}
|H(\xi,x)|
\leqslant
\| \Gamma \| _{C ( \T ^2 \times J )}^n\cdot|H(T^n(\xi,x))|
\to 0\quad\text{as }n\to\infty\ ,
\end{equation*}
as $\hat\varphi$ is everywhere differentiable with bounded derivative.
Hence 
\begin{equation*}
\hat\phi'(x)=\frac{\partial\hat\varphi}{\partial x}(\xi,x)=X_3(\xi,x,\hat\varphi(\xi,x))
\quad\text{for all }(\xi,x)\in\T^2\ ,     
\end{equation*}
 which
proves that $\hat\phi=\ellss_{(\xi,x,\hat\phi(x))}$ for all $(\xi,x)\in\T^2$, see
Definition~\ref{def:l_ss}. 
In particular, $\hat\phi$ is of class~$\CLip$.    
\end{proof}

\begin{proof}[Proof of Theorem~\ref{theo:dimension}]
a) Assume that
$\inf_{\nu\in\cM_T(\T^2),\, \nu\, {\rm Gibbs}}\lambda^*(\nu)>0$. Then
%we are in case (iii) of Proposition~\ref{prop:classification-general}, and
 case~B1 of Theorem~\ref{theo:B-class} is clearly excluded. Suppose for a contradiction that we are in case~B2. Then $\lambda^*(m^2)=\hat\lambda(m^2)>0$, and it follows from \cite[Lemma 2.69]{Otani2015} that $m^2(P)=0$ (observe also Footnote~\ref{foot:2.69} below).
Hence, invoking B2) again, the set $\cP$ is the graph of a $\CLip$-function $\hat\varphi$. Using the variational principle one shows then the first of the following identities:
 \begin{equation*}
 \begin{split}
 \inf\left\{\hat \lambda(\nu): \nu\in\cM_T(\T^2)\right\}
 &=
 \inf\left\{\hat \lambda(\nu): \nu\in\cM_T(\T^2),\, \nu\, {\rm Gibbs}\right\}\\
 &=
 \inf\left\{\lambda^*(\nu): \nu\in\cM_T(\T^2),\, \nu\, {\rm Gibbs}\right\}
>0\,.
 \end{split}
 \end{equation*}
 The semi-uniform ergodic theorem \cite{Sturman2000} now implies that there is $n_0\in\N$ such that $(f_\theta^{n_0})'(\hat\varphi(\theta))>2$ for all $\theta\in\T^2$, which implies readily that $P=\emptyset$.
  Therefore $\inf_{\nu\in\cM_T(\T^2),\, \nu\, {\rm Gibbs}}\lambda^*(\nu)>0$
  implies case~A of Theorem~\ref{theo:top-class} - in particular $P=\emptyset$ and $\dimH(P)=0$.\\[1mm]
b) Assume now that there is some $\nu\in\cM_T(\T^2)$ with $\lambda^*(\nu)<0$. Then there is also an ergodic measure with this property, and case~A of Theorem~\ref{theo:top-class} is excluded. 

Suppose first that $m^2(P)=1$. Then $\lambda^*(m^2)\leqslant0$ by Proposition~\ref{prop:classification-general}(i).
Of course $\dimH(P)=2$ in this case. As the supremum on the r.h.s. of (\ref{eq:dim-formula}) is bounded by $1$, and as $h_{m^2}(T)=h_m(\tau)=\int\log\tau'\,dm$ by Rohlin's formula, the supremum $1$ is indeed attained for $\nu=m^2$, and equality in (\ref{eq:dim-formula}) is proved with the supremum attained for $\nu=m^2$.

Suppose now that $m^2(P)=0$. Then we are in case~B2-ii of Theorem~\ref{theo:B-class},
 and, for each Gibbs measure $\nu\in\cM_T(\T^2)$, the graph of $\varphi^*$ coincides $\nu$-almost surely with the
graph of the $\CLip$-function $\hat\varphi$ from that theorem.
Denote $\hat\lambda(\nu)=\int\log f_\theta'(\hat\varphi(\theta))\,d\nu(\theta)$.
By Proposition~\ref{prop:classification-general}, $\hat\lambda(\nu)\leqslant\lambda^*(\nu)$ for each $\nu\in\cM_T(\T^2)$.

Note next that
\begin{equation}\label{eq:dim-with-two-varphi}
	\sup _{\nu\in\cM_T(\T^2)} \left\{  \frac{h _\nu (T)}{\int \log \tau ' \, d \nu ^{(2)}} : \hat\lambda ( \nu ) \leqslant 0 \right\} = \sup _{\nu\in\cM_T(\T^2)} \left\{  \frac{h _\nu (T)}{\int \log \tau ' \, d \nu ^{(2)}} : \lambda ^\ast ( \nu ) \leqslant 0 \right\} .
\end{equation}
Indeed, as $\hat\lambda(\nu)\leqslant\lambda^*(\nu)$ for all $\nu\in\cM_T(\T^2)$, 
we have l.h.s.\,$\geqslant$\,r.h.s. On the other hand, as $\theta=(\xi,x)\mapsto\log\tau'(x)$ and $\theta=(\xi,x)\mapsto\log f_x'(\hat\phi(x))$ are Hölder continuous, the l.h.s. is maximized by a Gibbs measure $\tilde{\nu}$, 
and that satisfies $\lambda^*(\tilde{\nu})=\hat{\lambda}(\tilde\nu)$, because $\varphi^*=\hat\varphi$ $\nu$-a.s.

It remains to show that
\begin{equation}
	 \dim _H ( P ) = \sup _{\nu\in\cM_T(\T^2)} \left\{  \frac{h _\nu (T)}{\int \log \tau ' \, d \nu ^{(2)}} : \lambda^* ( \nu ) \leqslant 0 \right\} + 1\,.
\end{equation}
The '$ \geqslant $'-estimate is precisely
Theorem 2.1 of \cite{Otani2015}.
To prove the reverse inequality, we first apply Lemma 2.69 of \cite{Otani2015}.
That lemma asserts\footnote{To be precise, \cite[Lemma 2.69]{Otani2015} asserts this property for $\varphi^*$ instead of $\hat\varphi$, but its proof applies verbatim also to $\hat\varphi$: only in the first line of the proof one occurence of $\varphi^*$ must be replaced by $\hat\varphi$.\label{foot:2.69}}
\[
	P \subseteq N^\leqslant_0:=\left\{ \theta : \liminf _{n \rightarrow \infty} \frac{1}{n} \sum _{i=0} ^{n-1} \log f _{T^{-j}\theta} ' ( \hat\varphi(T^{-j}\theta) ) \leqslant 0 \right\} .
\]
Therefore, in view of (\ref{eq:dim-with-two-varphi}), it remains to show that
\begin{equation}\label{eq:last-inequ}
\dimH(N^\leqslant_0)
\leqslant
\sup _{\nu\in\cM_T(\T^2)} \left\{ \frac{h _\nu (T)}{\int \log \tau ' \, d \nu ^{(2)}} : \hat\lambda  ( \nu ) \leqslant 0 \right\}  + 1\,.
\end{equation}
This can be deduced from the very general results in \cite{Barreira-book} in exactly the same way as done in the proof of Lemmas 4 and 5 in \cite{KeOt2012}. Then Lemma 5 of that paper asserts (among others) that $\dimH(N^\leqslant_0)=D(0)+1$, where\footnote{In\cite{KeOt2012}, the reader finds $\int-\log\|dT\mid E^s\|\,d\nu$ instead of $\int\log\tau'\,d\nu^{(2)}$. In our special setting where $T$ = baker map, we have $T^{-1}(\xi,x)=T(x,\xi)$, so that $\int-\log\|dT\mid E^s\|\,d\nu=
\int\log\|dT^{-1}\mid E^s\|\circ T^{-1}\,d\nu=\int\log\tau'(x)\,d\nu(\xi,x)=\int\log\tau'\,d\nu^{(2)}$.}
\begin{equation*}
D(0)
=
\sup _{\nu\in\cM_T(\T^2)} \left\{  \frac{h _\nu (T)}{\int \log \tau ' \, d \nu ^{(2)}} : \hat\lambda  ( \nu ) =0 \right\} \,.
\end{equation*}
From this the estimate (\ref{eq:last-inequ}) follows at once.
\end{proof}

\appendix

\section{Further proofs}

\begin{proof}[Proof of Corollary~\ref{coro:exponent}]
For $ y \in [ \varphi ^- ( \theta ) , \varphi ^+ ( \theta ) ] $, this follows immediately from Proposition \ref{prop:classification-general}.
For $ y \not \in [ \varphi ^- ( \theta ) , \varphi ^+ ( \theta ) ]  $, use the pull-back construction.
For example, for $ y >  \varphi ^+ ( \theta ) $, take a large $ a > 0 $ and define $ \psi ^+ _k ( \theta ) = f _{T ^{-k}\theta} ^k (a) $.
Then, in view of the monotonicity of the maps $f_\theta$, we have $ 0\leqslant f_\theta^k ( y ) - f_\theta^k ( \varphi ^+ ( \theta ) )  \leqslant  \psi ^+ _k(T^k\theta) - \varphi ^+(T ^k\theta) $ for all $ k \in \N _0 $.
Observe that by Birkhoff's ergodic theorem, we have for each $ N \in \N $
\[
	 \frac{1}{n} \sum _{k=0} ^{n-1} | f_\theta^k ( y ) - f_\theta^k ( \varphi ^+ ( \theta ) ) | \leqslant \frac{N }{n} \|  \psi ^+ _0 - \varphi ^+ \| _\infty + \frac{1}{n} \sum _{k=N} ^{n-1}  \sup _{j \geqslant N}(  \psi ^+ _j - \varphi ^+ )(T ^k\theta) \rightarrow \int \sup _{j \geqslant N}(  \psi ^+ _j - \varphi ^+ ) \, d \nu 
\]
for $ \nu $-a.a. $ \theta $ as $n\to\infty$.
In the limit  $ N \rightarrow \infty $, it follows
\[
	\lim _{n \rightarrow \infty }\frac{1}{n} \sum _{k=0} ^{n-1} | f_\theta^k ( y ) - \varphi ^+ (T^k \theta ) |
	=
	\lim _{n \rightarrow \infty }\frac{1}{n} \sum _{k=0} ^{n-1} | f_\theta^k ( y ) - f_\theta^k ( \varphi ^+ ( \theta ) ) | = 0 
\]
for $ \nu $-a.a. $ \theta $.
As $(\theta,y)\mapsto\log f_\theta'(y)$ is continuous, this implies $ \lambda _{(\theta , y)} = \lambda ^+ ( \nu ) $.
\end{proof}

\begin{proof}[Proof of Lemma \ref{lemma:shortX3}]
Let
\begin{equation}
	X _3 ( \xi, x , y ) := - \sum _{k=0} ^\infty  \Gamma ^k ( \xi, x, y) \cdot  A \circ f ^k ( \xi, x , y ) ,	\label{eq:X_3}
\end{equation}
where $ \Gamma ^k := \prod _{j=0} ^{k-1} \Gamma \circ f ^j $,
\[
	\Gamma ( \xi, x , y ) := \frac{  \sigma(\xi)  }{f _x ' (y)} \quad \mbox{and} \quad	
 A(\xi ,x,y):=
%	:= A (x, y) := 
%	\frac{\beta (x, y)}{f _x ' (y)} = 
\frac{\frac{\partial}{\partial x}f_x(y)}{f _x ' (y)}  .
\]
Since $ \| \Gamma \| _{C ( \T ^2 \times J )} < 1 $ due to Hypotheses \ref{hyp:weak-contraction}, $ X _3 $ is well-defined and uniformly bounded.
A straightforward calculation shows 
 \begin{equation}\label{eq:equivariance}
 	D f _{(\xi, x, y)} \, \left[ \begin{matrix}
 	0  \\ 1 \\ X _3 (\xi, x, y) \end{matrix} \right] = 
 	\sigma(\xi)  
 	\, \left[ \begin{matrix}  
 	0  \\ 1 \\ X _3 ( f (\xi, x, y) )
 	\end{matrix}  \right],
 \end{equation}
i.e. $X$ is $DF$ equivariant,
and a closer inspection of the definition of $X_3$ reveals that $(x,y)\mapsto X_3(\xi,x,y)$ is continuous for each fixed $\xi\in \T$. 
For use in the proof of Theorem~\ref{theo:B-class}, note that \eqref{eq:equivariance} implies
\begin{equation}\label{eq:equivariance3}
{X}_3(\xi,x,y)
=
-A(\xi,x,y)	+\Gamma(\xi,x,y)\cdot  X_3(f(\xi,x,y))\ .
\end{equation}

We turn to the proof of
%It remains to prove 
the uniform bound (\ref{eq:X3deriv-y-bound}) on $\partial X_3/\partial y$:
Let $ \sigma (\xi , x ) := \sigma(\xi) $    
and $ q (\xi, x, y)
%:= q ( x, y )
:= 1 / f _x ' (y) $.
Recall the notations $ \sigma^k := \prod _{j=0} ^{n-1} \sigma \circ T ^j $, $ \Gamma ^k := \prod _{j=0} ^{k-1} \Gamma \circ f ^j $ and $ q ^k := \prod _{j=0} ^{k-1} q \circ f ^j $.

Since $ \Gamma ^k ( \xi, x, y ) = \sigma^k (\xi, x) \, q ^k (\xi, x, y) $ and $ q ^j (\xi, x, y) = 1 / (f ^j _{(\xi, x)}) ' (y) $, we have
\begin{equation*}
% 	\frac{\frac{\partial}{\partial y} \Gamma ^k ( \xi, x, y ) }{\Gamma ^k ( \xi, x, y )}	 
% 	=
 	\frac{\partial}{\partial y}\log \Gamma ^k ( \xi, x, y ) 
= 	
\sum_{j=0}^{k-1}\frac{\partial}{\partial y}\left(\log q \circ f^j( \xi, x, y )\right)  	
 	=
\sum _{j=0} ^{k-1} \left(\frac{\partial}{\partial y}\log q \right) \left( 
f^j(\xi,x,y) \right) \cdot ( f _{(\xi,x)} ^j )' (y)\  ,
\end{equation*}
so that
\begin{equation*}
	\left| \frac{\partial}{\partial y} \log\left(  \Gamma ^k ( \xi, x, y) \right) \right| 
	 =  \left| \sum _{j=0} ^{k-1} \frac{\partial \log q}{\partial y} \left(f^j(\xi,x,y)   \right) \cdot ( f _{(\xi,x)} ^j )' (y)  \right| 
	 \leqslant 
	 C_0     \ \sum _{j=0} ^{k-1} ( f _{(\xi,x)} ^j )' (y) ,
\end{equation*}
where the constant $C_0 := \|\partial(\log f_x')/\partial y\| _{C ( \T \times J )} $
is finite by Hypothesis~\ref{hyp:special-ass} and~\ref{hyp:regularity}.
Furthermore,
with the constants
$C:=C_0\cdot\|A\|_{C(\T\times J)}$ and $C':=\|\frac{\partial A}{\partial y}\|_{C(\T\times J)}$, that are also finite,
\begin{eqnarray*}
	&& 
	\sum _{k=0} ^\infty  \left| \frac{\partial}{\partial y} \left(  \Gamma ^k ( \xi, x, y) \cdot  A \circ f ^k ( \xi, x , y ) \right) \right| \\
	& \leqslant  &	
	\sum _{k=0} ^\infty  \Gamma ^k ( \xi, x, y ) \left( C_0 \cdot |A\circ f ^k| ( \xi, x , y ) \sum _{j=0} ^{k-1} ( f _{(\xi,x)} ^j )' (y) +\bigg|\frac{\partial A}{\partial y}\circ f^k\bigg|(\xi,x,y)\cdot (f^k_{(\xi,x)})'
	(y)  \right)\\
	& \leqslant & 
	C \sum _{j=0} ^\infty  \sum _{k=j} ^\infty \Gamma ^k ( \xi, x, y )  \; ( f _{(\xi,x)} ^j )' (y) 
	+C'\sum_{k=0}^\infty \Gamma^k(\xi,x,y)\cdot(f^k_{(\xi,x)})'(y)
	\\
	& \leqslant &	
	\frac{C }{1 - \| \Gamma \| _{C ( \T ^2 \times J )}} \sum _{j=0} ^\infty  \Gamma ^j ( \xi, x, y )  \; ( f _{(\xi,x)} ^j )' (y) 
	+C'\sum_{k=0}^\infty \Gamma^k(\xi,x,y)\cdot(f^k_{(\xi,x)})'(y)
	\\
	& = &	
	\left(\frac{C}{1 - \| \Gamma \| _{C ( \T ^2 \times J )}}+C'\right)   
	 \sum _{j=0} ^\infty \sigma ^j ( \xi, x )              \\
	& \leqslant &	
	\frac{C+C' }{(1 - \| \Gamma \| _{C ( \T ^2 \times J )}) ( 1 - \| \sigma \| _{C ( \T^2 )} )} 
\end{eqnarray*}
for all $ (\xi, x, y) \in \T ^2 \times I $.
The differentiability and the boundedness of $\partial X_3/\partial y$ follow from this.

The proof of the uniform bound \eqref{eq:X3deriv-x-bound} is more tedious, but it goes along the same lines. We leave it to the reader and note only that, without this bound,
all $\CLip$-assertions in Section~\ref{sec:results} would have to be replaced by $C^1$.
\end{proof}

\end{document}